\documentclass[journal]{IEEEtran}

\usepackage{tikz}
\usepackage{amsfonts}
\usepackage{scrextend}
\usepackage{amsfonts}
\usepackage{amsmath,bm}
\usepackage{amssymb}
\usepackage{enumitem}
\usepackage{mathtools}
\usepackage{amsthm}
\usepackage{breqn}
\usepackage{dsfont}
\usepackage{sgame}

\allowdisplaybreaks

\newtheorem{theorem}{\bf{Theorem}}

\newtheorem{lemma}{\bf{Lemma}}

\newtheorem{definition}{\bf{Definition}}

\newtheorem{assumption}{Assumption}

\renewcommand{\P}{\text{Pr}}
\newcommand{\BR}{\text{BR}}
\newcommand{\uniform}{\text{Unif}}

\usepackage[linesnumbered,boxruled]{algorithm2e} %see section 7, esp subsection 7.3, from the algorithm2e ctan page for details. i like boxruled, but ruled, boxed are other options.... option longend prints stuff like endfor, option shortend (default) prints only end, and option noend doesn't print an end. i like this one, personally. 

\begin{document}

\title{   
Decentralized Learning for Optimality in Stochastic Dynamic Teams and Games with Local Control and Global State Information
\thanks{A conference version \cite{YAYCDC19} was presented at the 2019 Conference on Decision and Control and serves as an announcement of the partial results presented here without details. The conference version \cite{YAYCDC19} does not contain the results on weakly acyclic games or any of the proofs presented here.}
\thanks{This research was supported in part by the Natural Sciences and Engineering Research Council (NSERC) of Canada.}}

\author{Bora~Yongacoglu, G\"urdal Arslan, and  Serdar Y\"uksel
\thanks{B. Yongacoglu and S. Y\"uksel are with the Department of Mathematics and Statistics, Queen's University, Kingston, ON, Canada, email: \{1bmy,yuksel@queensu.ca\}.
G. Arslan is with the University of Hawaii. Email: \{gurdal@hawaii.edu\}.}}

\maketitle

\begin{abstract}

Stochastic dynamic teams and games are rich models for decentralized systems and challenging testing grounds for multi-agent learning. %There are several existing algorithms for learning in stochastic games, including some with guarantees of convergence to equilibrium. However, there may be high cost suboptimal equilibria in cooperative games where team optimality is the goal. 
Previous work that guaranteed team optimality assumed stateless dynamics, or an explicit coordination mechanism, or joint-control sharing. In this paper, we present an algorithm with guarantees of convergence to team optimal policies in teams and common interest games. The algorithm is a two-timescale method that uses a variant of Q-learning on the finer timescale to perform policy evaluation while exploring the policy space on the coarser timescale. Agents following this algorithm are ``independent learners": they use only local controls, local cost realizations, and global state information, without access to controls of other agents. The results presented here are the first, to our knowledge, to give formal guarantees of convergence to team optimality using independent learners in stochastic dynamic teams and common interest games.	

\end{abstract}


\IEEEpeerreviewmaketitle

\section{Introduction} \label{sec:intro}

In modern control engineering applications, two challenges are becoming increasingly common:  online problems and decentralization. In online problems, the system to be controlled is not initially known by the agent and must be learned. In decentralized systems, several autonomous decision-makers act in a shared environment. This paper is concerned with multi-agent reinforcement learning (MARL), which is at the intersection of these two challenges. We use stochastic games to model the shared environment, and we present algorithms suitable for stochastic dynamic teams under a particular decentralized information structure.

In online problems, important knowledge of the system to be controlled is initially unavailable to the controller. Classical methods for solving control problems, such as linear programming, dynamic programming, and convex analytic methods, cannot be implemented without access to the system model. Instead, the control agent must use observed feedback to learn control policies. Reinforcement learning has had considerable success in single-agent control problems, both in applications and in theory, where methods such as Q-learning \cite{watkins,watkinsDayan,tsitsiklis} recover optimal policies when used in a stationary environment.

A second challenge comes from decentralization. Decentralized systems are characterized by multiple agents acting in a common environment with some local information available to each. The costs incurred by one agent in a decentralized system depend, in general, on its own actions, the actions of other agents, and the history of the system. Such coupled interactions are common in complex, real-world engineering applications. Some examples of systems that are inherently decentralized are sensor networks, stochastic networked control systems, Internet of Things, and energy systems.

Compared to the success of reinforcement learning in stationary single-agent problems, there are relatively few formal results on MARL. This is partly explained by the loss of stationarity: when multiple learning agents interact, a given agent will change its behaviour to exploit learned information. From the point-of-view of the remaining agents, this agent is a part of the environment, and so the environment is non-stationary \cite{hernandez2017survey}. Consequently, one of the fundamental assumptions made for single-agent theory does not hold in MARL, and theoretical guarantees do not carry over.

Stochastic games \cite{fink1964equilibrium,altman1996non,ding2013stochastic,6464723} generalize both repeated games \cite{fudenberg1991game} and Markov decision problems (MDPs). Like repeated games, players in stochastic games must be strategic and respond to the policies used by other agents. Unlike repeated games, in which the same stage game is played at every time step, the stage game played at a given time in a stochastic game depends on the history of play, which is summarized by a state. As in MDPs, agents in stochastic games must select actions with the state process and its long-term cost implications in mind. As stochastic games provide a rich model for dynamic, strategic decision making, they are a popular framework for studying MARL \cite{littman1994markov}.

Stochastic dynamic teams \cite{ho1980team,yukselBasar} and common-interest games  \cite{AUMANN19895,TAKAHASHI2005231} model cooperative systems and so are of special interest to decentralized control. In teams, all players incur the same stage costs and interests are perfectly aligned. Common interest games generalize teams in a natural way: in common interest games, agents to not necessarily incur identical stage costs, but there are a subset of joint policies which each agent strictly prefers to all other policies. Despite the incentive to coordinate behavior in common interest games, coordination is generally difficult in online problems when information is decentralized.

As we will outline in detail in Section \ref{sec:LitReview}, there are relatively few theoretical results for stochastic games without control-sharing. Even when assuming full state observability at each agent rather than the more general assumption of partial state observability, there are no rigorous results that guarantee team optimality in truly stochastic teams without relying on control-sharing.

\subsection{Contributions} In this paper, we present a decentralized learning algorithm for playing stochastic common interest games, a class of games which model decentralized control problems and contain stochastic teams as a special case. We give formal guarantees of convergence to team optimal policies without use of control sharing among agents.

\begin{itemize}
\item[(i)] In Theorem~\ref{theorem1}, we consider stochastic common interest games and introduce an algorithm (that only uses local cost and local action history and the common state of the system) that provably converges to a team optimal policy in a probabilistic sense that is made precise in the theorem. What makes this algorithm different from our prior work \cite{AY2017}, which guaranteed convergence to equilibrium but not team optimal policies, is the utilization of a finite window of the most recent (noisy) aggregate cost scores to adaptively estimate the lowest possible cost for each decision maker. %\sout{Unlike much of the relevant work, this algorithm does not require control action sharing, and each decision maker can be completely oblivious to the presence of other decision makers in the system.}
\item[(ii)] Theorem~\ref{weaklyTheorem} considers a specific implementation of our main algorithm in the context of weakly acyclic games. We show this algorithm leads to equilibrium policies in weakly acyclic games and, furthermore, if the game is also a common interest game then play will settle to a team optimal policy. This theorem strengthens one of the main results from \cite{AY2017}.
\item[(iii)] In Theorem \ref{th:constasp}, we obtain convergence to team optimality in the stronger sense of almost sure convergence by using constant, preset aspiration levels. This result requires a stronger assumption that the preset aspiration levels separate the team optimal policies from the other policies. Theorem~\ref{th:constasp} also describes the long-run behaviour of the algorithm with constant aspirations when used in a general stochastic game.
%\item[(iv)] In Theorem~\ref{th:rg}, we consider the simpler case of repeated games where the same stage game is repeated by decision makers with no look ahead noise-free cost readings, and obtain more general and stronger convergence results than those in Theorem~\ref{theorem1} by using the minimum cost realizations in the entire past as the estimates of the lowest possible cost for each decision maker.
\end{itemize}

\noindent{}These contributions are the first formal guarantees of achieving team optimality in stochastic common interest games under full state observability but no action sharing.

%\begin{center} \textbf{Paper Organization} \end{center}

The remainder of the paper is organized as follows: Section \ref{sec:LitReview} surveys related literature. In Section \ref{sec:Model}, we specify the stochastic game model and provide relevant background. In Section \ref{sec:mainAlgo}, we present our main algorithm and state Theorem \ref{theorem1}. Section \ref{sec:weaklyAcyclic} presents Theorem \ref{weaklyTheorem}, which strengthens a result from \cite{AY2017}. Section \ref{sec:constasp} considers a variant of the main algorithm and presents Theorem \ref{th:constasp}, which studies the variant algorithm's long-term behaviour in general sum games. Section \ref{sec:simulation}  contains numerical results from a simulation study, and Section~\ref{sec:conclusion} concludes the paper. The proofs of our main technical results are contained in the appendices.

\subsection{Notation} \label{ss:notation} $\mathbb{R}$ denotes the real numbers, $\mathbb{N}$ and $\mathbb{N}_+$ denote the nonnegative and positive integers, respectively. $\P(\cdot)$ and $E(\cdot)$ denote the probability and the expectation, respectively. For a finite set $S$, $\mathcal{P}(S)$ denotes the set of probability distributions over $S$. For finite sets $S, S'$, we let $\mathcal{P} ( S' | S)$ denote the set of stochastic kernels on $S'$ given $S$. An element $\mathcal{T} \in \mathcal{P} ( S' | S)$ is a collection of probabilities distributions on $S'$, with one distribution for each $s \in S$, and we write $\mathcal{T} ( \cdot | s )$ for $s \in S$ to make this distributional dependence on $s$ explicit. We write $Y \sim f$ to denote that the random variable $Y$ has distribution $f$. If the distribution of $Y$ is a mixture of other distributions, say with mixture components $f_i$ and weights $p_i$ for $1 \leq i \leq n$, we write $Y \sim \sum_{i = 1}^n p_i f_i$. The Dirac distribution concentrated at $x \in \mathbb{R}$ is denoted $\mathbb{I}_x$. For a finite set $S$, $\uniform(S)$ denotes the uniform distribution over $S$ and $2^S$ denotes the set of subsets of $S$. $(x)^+:=\max\{x,0\}$, for $x\in\mathbb{R}$.

\section{Literature Review} \label{sec:LitReview}

Interest in using single-agent reinforcement learning in multi-agent environments dates back at least as far as \cite{tan1993multi}, in which Q-learning is studied in a cooperative predator-prey simulation. In \cite{sen}, multiple agents run Q-learning in a block-pushing task without sharing actions with one another, and the authors suggest that cooperative behaviour may emerge even without explicit communication between agents.

In addition to presenting empirical results and formal conjectures, an important terminological distinction was popularized in \cite{clausBoutilier}, where the authors distinguish between joint action learners and independent learners: joint action learners use the past actions of all agents in their learning, while independent learners use only local action histories.

Early rigorous results on MARL in games was concerned mostly with joint-action learners. In \cite{littman1994markov}, Littman proposed stochastic games as a framework for studying MARL and presented the \textit{Minimax Q-learning} algorithm, a joint-action learner designed for two-player zero-sum games. Convergence results for this method were proved in \cite{littman1996generalized}.  The main idea from \cite{littman1994markov} was extended in \cite{hu2003nash} and \cite{hu1998multiagent}, which present \textit{Nash Q-learning}, another joint-action learner with convergence guarantees under certain restrictive assumptions. Further contributions in this line include \textit{Friend-or-Foe Q-learning} \cite{littman2001friend}, \textit{Team Q-Learning} \cite{littman2001team}, and several others, e.g. \cite{greenwald2003correlated,tesauro2004hyper}. A considerably different approach is taken in \cite{wangSandholm}, which presents \textit{Optimal Adaptive Play} (OAP), a joint-action learner based on adaptive play \cite{young} rather than on Q-learning. OAP is shown to converge to a team optimal policy when used in a stochastic team.

Though early rigorous work focused on joint action learners, there has also been persistent interest in independent learners. As the number of joint actions is exponential in the number of agents, the computational burden of a joint action learner at any one agent becomes intractable for problems of even a moderate size. Scalability, robustness, and faster convergence are potential advantages of independent learners over joint action learners \cite{matignon2012Survey, stone2000survey}. The applicability of the set-up considered here and other advantages are covered in greater detail in \cite{matignon2012Survey} and \cite{leottau2018decentralized}. For a recent survey of MARL that discusses other decentralized set-ups, see \cite{zhang2019survey}.

\textit{Distributed Q-learning}, an independent learner designed for teams, was presented in \cite{lauer2000algorithm}, along with a guarantee of convergence to team optimality in teams with deterministic state dynamics and costs. When using this algorithm, an agent only updates its Q-factors when an improvement is observed, attributing unfavourable feedback to its teammates' experimenting with other actions. This optimistic approach leads to poor performance in problems with random state transitions or cost readings \cite{matignon2012Survey}.

An algorithm called \textit{Win or Learn Fast Policy Hill Climbing} (WoLF-PHC) was introduced in \cite{bowling2002multiagent}. An agent using WoLF-PHC selects actions according to an exploration policy and iteratively improves its exploration policy using its learned Q-factors by updating toward a best-response. Although no formal results are presented for stochastic games, the key innovation of \cite{bowling2002multiagent} is its policy update: the agent compares the performance of the current exploration policy to that of a distinguished ``average policy." When the current policy outperforms the average policy, the agent changes its policy relatively slowly; when the current policy is underperforming, the agent changes its policy more rapidly.

Following \cite{lauer2000algorithm} and \cite{bowling2002multiagent}, a number of algorithms based on Q-learning were proposed for stochastic games. Some of these algorithms, such as \textit{Hysteretic Q-learning} \cite{matignon2007hysteretic}, modify the Q-factor update. Other methods, including the \textit{Frequency Maximum Q} heuristic presented in \cite{kapetanakis2002reinforcement} and its extensions to stochastic games \cite{matignon2009coordination}, modify action selection. Still other methods, such as lenient learning \cite{panait2006lenient,wei2016lenient}, modify both the Q-factor update as well as the action selection mechanism in an attempt to achieve optimality in cooperative games. With the exception of \cite{lauer2000algorithm} described above, these works offer only empirical support for their algorithms, rather than formal guarantees of convergence to team optimality. For a survey on this line of research and a description of obstacles in MARL, see \cite{matignon2012Survey}. %{\color{red}We note that all of the preceding studies---both independent and joint action learners alike---studied models with full state observations at each agent.}%Just say this in the response letter instead.

While researchers in the machine learning community sought empirically successful algorithms for stochastic games, a parallel line of research in the control and operations research communities sought rigorous results in the more restricted class of stateless repeated games.

Among the literature on MARL for repeated games, \cite{chasparis2013aspiration,marden2009payoff} and \cite{marden2014achieving} are most relevant to this paper. Although the algorithms and analysis presented in these works differ from one another, each operates using the principle of exploring an agent's set of actions more aggressively when the agent perceives it is underperforming.

Reference \cite{marden2009payoff} presents three algorithms, including \textit{Safe Experimentation Dynamics}, which is shown to lead to team optimality in repeated teams with high probability. Using this method, an agent maintains a baseline action and baseline cost while experimenting with other actions. Each time an action is taken, its immediate cost is compared with the baseline cost; the baseline cost is adjusted when a lower cost is observed, and the action achieving this lower cost becomes the new baseline.

Agents using the algorithm from \cite{marden2014achieving} maintain a binary ``mood" variable, which is meant to capture whether the agent is content with its current performance. It is shown that all stochastically stable outcomes maximize the sum of joint payoffs across all agents.

\textit{Aspiration learning} for repeated coordination games is presented in \cite{chasparis2013aspiration}, along with formal results on the stochastic stability of efficient outcomes. An agent using this algorithm iteratively sets its aspiration level, a scalar threshold value that represents the highest cost (or lowest reward) that the agent finds acceptable. When receiving costs higher than its aspiration level, the agent is unsatisfied and explores alternative actions more aggressively.

Other work in this area includes \cite{marden2012revisiting,PRADELSKI2012882} and \cite{GEB:stochimitation:2014}. Variants of log-linear learning for repeated games were studied in \cite{marden2012revisiting, PRADELSKI2012882} and come with guarantees on the stochastic stability of efficient outcomes. The stochastic imitation dynamics introduced in \cite{GEB:stochimitation:2014} assign probability one to efficient outcomes in large class of repeated games.

One explanation for the greater number of rigorous results on independent learners in a repeated game setting is the lack of state dynamics. In repeated games, the same stage game is played in each period and there is no tradeoff between short- and long-term costs. As such, the scalar cost realizations can be used directly when setting aspiration levels (as in \cite{chasparis2013aspiration}) or baseline costs (as in \cite{marden2009payoff} and \cite{marden2014achieving}). In contrast, policy evaluation is inherently slow (due to delayed rewards), noisy, and algorithm dependent in games with random state dynamics, and this is only exacerbated by the presence of other learning agents. Consequently, extending the preceding methods is a significant challenge.

In this paper, we study stochastic teams and common interest games with full state information at each agent but no action sharing between agents. This set-up arises naturally in problems where the state can be sensed by a global sensor and broadcast to agents. In \cite{matignon2010designing}, a (physically) distributed array of micro-electro-valves producing controlled and directed micro-air-jets is used to steer the motion of a small object on a smart surface. The state of this system is the current and previous positions of the object which is sensed by an overhead camera and accessed by all control units each controlling a separate valve. Each control unit implements a standard Q-learning algorithm based on the global state and its own control observations (by ignoring the other control units) for reasons stated as follows: ``A fully centralized control architecture is not suitable due to processing complexity and the number of communication channels required''.
%Lighting coverage over a sensor network is considered in \cite{tham2005multi}, where it is assumed that the global state is broadcast by radio to each agent in the network.
%In other applications, this information structure arises due to computational considerations.
In \cite{leottau2018decentralized}, robotics problems involving multi-dimensional action spaces are considered. The authors observe that centralized approaches in problems with multiple actuators are often intractable due to a combinatorial explosion of the joint state-action space. Among other decentralization schemes, the authors consider the case with full state but only local actions, wherein the actuators are able to sense the global state variable (e.g. two dimensional position and velocity in a vehicle navigation problem; three dimensional position in a joint manipulation task)  but do not attempt to sense one another's actions for computational tractability.
Other applications for which this information structure is appropriate include problems where the state variable is a commonly observed price as well as problems in traffic networks, where link latencies can be broadcast using a mobile application. 

Another motivation for studying this set-up is that algorithms designed for problems with full state information but no action-sharing have been successful even when used in problems possessing a different information structure, such as partial state observability. Examples of studies that use partial state observations as a surrogate for complete state observations and then use methods designed for our information structure include interference control in wireless networks \cite{galindo2010distributed, nie2016q} and cache placement in wireless networks \cite{lin2019marl}. Many further examples can be found in the area of cognitive radio; see \cite{wang2016survey} and the references therein.

%Finishing sentence: this paper presents the first independent learner for stochastic teams (and their generalizations) that come with a rigorous guarantee of convergence to team optimality. Don't forget to cite Kaiqing Zhang's two papers.

In \cite{AY2017}, we introduced an independent learner that provably leads to equilibrium in weakly acyclic stochastic games in general and in teams in particular. However, stochastic teams generally have both team optimal equilibrium policies and suboptimal equilibrium policies, and suboptimal equilibria can perform arbitrarily worse than an optimal equilibrium. A simple but illustrative example is offered in Section \ref{sec:Model}. Thus, guarantees of finding an equilibrium joint policy are not satisfactory in the context of decentralized control when cost minimization is a design goal. In this paper, we modify the main algorithm from \cite{AY2017} to guarantee convergence to team optimality when possible. \\

\section{Background} \label{sec:Model}

\subsection{Stationary Markov Decision Problems and Q-learning}

A stationary Markov Decision Problem (MDP) with a discounted cost criterion is a discrete time process characterized by the following:
\begin{enumerate}
	\item A finite set of states $\mathbb{X}$
	\item A random initial state $x_0 \in \mathbb{X}$
	\item A finite set of control actions $\mathbb{U}$
	\item A discount factor $\beta \in (0,1)$
	\item A cost function $c: \mathbb{X} \times \mathbb{U} \to \mathbb{R}$
	\item A transition probability kernel $P\in \mathcal{P} ( \mathbb{X} | \mathbb{X} \times \mathbb{U} )$ for determining the next state given the current state-action.
\end{enumerate}

At time $t\in\mathbb{N}$, the system is in state $x_t \in \mathbb{X}$ and the decision maker\footnote{We use the terms agent, decision maker, and player interchangeably.} (DM) selects a control action $u_t \in \mathbb{U}$. The DM then incurs a stage cost $c(x_t, u_t)$, and the system randomly transitions to the next state, $x_{t+1}$, according to the probability distribution $P( \cdot \vert x_t, u_t)$. %The agent's objective is to minimize the expectation of their series of discounted costs:
%\[
%E [ \sum_{t = 0}^\infty \beta^t c ( x_t, u_t) ]
%\]
%Here, the subscript $\pi$ is used to show that $\pi$ is the policy being used. The policy used determines a probability measure on the sequence of states and control actions. ... \bmy{I have removed the talk of policies and introduced it afterward.}
We assume that, prior to selecting $u_t$ at time $t\in\mathbb{N}$, the DM has access to the information $I_t$ defined by
\[
I_0=\{x_0\}, \quad I_{t+1} = I_t \cup \{ x_{t+1}, u_t, c ( x_t, u_t ) \}, \ t\in\mathbb{N}.
\]

A policy is a rule for selecting control actions based on the information available. In principle, the DM may use any function of $I_t$ to choose $u_t$, possibly with randomization. Fixing a policy $\theta$ induces a probability distribution on the sequence of state-actions $\{ (x_t, u_t )\}_{t \in \mathbb{N}}$. This induced probability measure is used to define the cost criterion:

\[
J_x (\theta) := E^\theta \left( \sum_{t\in\mathbb{N}} \beta^t c ( x_t, u_t) \Big| x_0 = x   \right), \quad \forall x\in\mathbb{X}
\]
where $E^\theta$ denotes that the stochastic process $\{ (x_t, u_t ) \})_{t \in \mathbb{N}}$ is determined by the policy $\theta$.

The DM's goal is to select a policy that minimizes the cost functional $J_x$ in every initial state $x \in \mathbb{X}$. Although the agent can use an arbitrarily complicated, history dependent policy, it is well-known (see, for example, \cite{HernandezLermaMCP}) that this minimum can be achieved within the simpler set of stationary randomized policies, which we identify with the set $\Delta = \mathcal{P}( \mathbb{U} | \mathbb{X} )$. A stationary randomized policy $\theta \in \Delta$ uses only the most recent state $x_t$ to (randomly) select an action $u_t$ in a time-invariant manner; that is, when the agent follows a policy $\theta \in \Delta$, we have $u_t \sim \theta ( \cdot | x_t )$. Within $\Delta$, we can further restrict our attention (without loss of optimality \cite{HernandezLermaMCP}) to the set of stationary deterministic policies $\Pi$, which we identify as $\Pi = \{ \pi : \mathbb{X} \to \mathbb{U} \}$. An agent following a policy $\pi \in \Pi$ selects its action as a deterministic function of the state, and we write $u_t = \pi ( x_t)$ or $u_t \sim \mathbb{I}_{\pi(x_t)}$. 

When the cost function and transition kernel are known, iterative methods such as value iteration can be used to obtain an optimal policy. Otherwise, model-free reinforcement learning techniques such as Q-learning \cite{watkins} can be used to recover an optimal policy.
In standard Q-learning, the DM begins with arbitrary Q-factors $Q_0 \in \mathbb{R}^{\mathbb{X} \times \mathbb{U}}$ and updates its Q-factors as follows:
\begin{align*}
Q_{t+1} (x_t, u_t) = & (1 - \alpha_t (x_t, u_t)) Q_t ( x_t, u_t)  \\
                  &  + \alpha_t (x_t, u_t) (c(x_t, u_t) + \beta \min_{v \in \mathbb{U}} Q_t ( x_{t+1}, v ) )\\
Q_{t+1} ( x, u  ) = & Q_t ( x, u  ), \quad \forall  (x, u) \not= (x_t, u_t)
\end{align*}
where $\alpha_t ( x_t, u_t )\in[0,1]$ is the step-size at time $t\in\mathbb{N}$. If all state-action pairs are visited infinitely often and the step-sizes vanish properly, then $\P (Q_t \rightarrow Q^*)=1$, where $Q^*$ is the vector of optimal Q-factors, the unique solution of a Bellman fixed point equation \cite{watkinsDayan}, \cite{tsitsiklis}.

Once $Q^*$ is attained, one can recover the value function $V^*$, using $V^* (x) = \min_{u \in \mathbb{U}} Q^*(x,u)$, or an optimal policy $\pi^*$, using $\pi^*(x) \in\arg\min_{u \in \mathbb{U}} Q^*(x, u)$. Moreover, learned Q-factors can be exploited during play: \cite{singh2000convergence} presents a Q-learning algorithm in which the DM's action selection converges to that of an optimal policy.

The popularity of Q-learning in stationary MDPs is justified: it is easy to implement and asymptotically recovers an optimal policy. However, this theoretical guarantee is predicated on the stationarity of the system. When a state-action $(x_t, u_t)$ is visited, the feedback received (in the form of a cost $c(x_t, u_t)$ and next state $x_{t+1}$) is always generated by the same Markovian source. If the system is not stationary, then convergence to the Q-factors $Q^*$ is not guaranteed.

%%%%%%%%%%%%%%%%%%%%%%%%%%%%%%%%%%%%%%%%%%%%

\subsection{Stochastic Games and Decentralized Q-learning}

A finite (discounted) stochastic game is a multi-agent generalization of a stationary MDP, and is characterized by
\begin{enumerate}
\item $N \in\mathbb{N}_+$ decision makers, the $i^{th}$ denoted by DM$^i$
\item A finite set of states $\mathbb{X}$
\item A random initial state $x_0 \in \mathbb{X}$
\item For each DM$^i$:
\begin{itemize}
\item[] A finite set of control actions $\mathbb{U}^i$
\item[] A discount factor $\beta^i \in (0,1)$
\item[] A cost function $c^i: \mathbb{X} \times \mathbb{U} \to \mathbb{R}$, where $\mathbb{U} :=  \times_{i = 1}^N \mathbb{U}^i$
\end{itemize}
\item A transition probability kernel $P \in \mathcal{P} ( \mathbb{X} | \mathbb{X} \times \mathbb{U} )$ for determining the next state given the current state and joint action.
%\item A transition probability kernel $P(\cdot \vert x, \textbf{u})$ for determining the next state given the current state-joint action $(x,\textbf{u}) \in \mathbb{X} \times \mathbb{U}$.
\end{enumerate}

At time $t\in\mathbb{N}$, the system is in state $x_t$, and each DM$^i$ chooses a control action $u^i_t$. While DM$^i$ only selects $u^i_t$, its incurred cost is given by $c^i ( x_t, \textbf{u}_t)$, where $\textbf{u}_t:=(u^1_t, \dots, u^N_t)$. Following the play of this stage game, the system randomly transitions to state $x_{t+1}$ according to $P ( \cdot \vert x_t, \textbf{u}_t)$. We consider the situation in which DM$^i$ observes only the state variable, its own actions, and its own cost realizations (DM$^i$ need not know the functional form of its cost). More precisely, prior to selecting $u_t^i$ at time $t\in\mathbb{N}$, DM$^i$ has access to the information $I_t^i$ defined by
\[
I_0^i=\{x_0\}, \quad I_{t+1}^i = I_{t}^i \cup \{ x_{t+1}, u_{t}^i, c^i ( x_{t}, \textbf{u}_{t})\}, \ t\in\mathbb{N}.
\]
In particular, DM$^i$ cannot see the past actions of the other DMs, $u^j_s$, for any $j \not= i$, $s\in\mathbb{N}$. This is in contrast to previous works such as \cite{wangSandholm}, \cite{hu2003nash}, \cite{littman2001friend} and \cite{zhangEtAl}.

A policy for DM$^i$ is a rule for selecting the sequence of local actions given the information available to DM$^i$. As in MDPs, DM$^i$'s goal is to minimize its long-term expected discounted cost. Unlike MDPs, however, DM$^i$'s cost is affected by the control actions of the other agents. We again restrict our attention to stationary randomized policies, which will be justified below. We denote the set of stationary randomized policies for DM$^i$ by $\Delta^i := \mathcal{P} ( \mathbb{U}^i | \mathbb{X})$, and similarly we use $\Pi^i =\{ \pi^i : \mathbb{X} \to \mathbb{U}^i \}$ to denote the set of stationary deterministic policies for DM$^i$.

We use boldface symbols to denote joint objects, i.e. lists of objects with one entry per agent, and we omit the agent superscript. The set of stationary joint policies is thus denoted by $\bm{\Delta} := \times_{i = 1}^N \Delta^i$ and the set of stationary deterministic joint policies is denoted $\bm{\Pi} := \times_{i = 1}^N \Pi^i$.

For notational convenience, we will use the agent superscript $-i$ to refer to a joint quantity for which DM$^i$'s position has been removed. Using this standard convention, the set of stationary joint policies for all agents except DM$^i$ is denoted $\bm{\Delta}^{-i} := \times_{j \not= i } \Delta^j$. Similarly, $\bm{\Pi}^{-i} = \times_{j \not= i} \Pi^j$ and $\mathbb{U}^{-i} = \times_{j \not= i} \mathbb{U}^j$. By convention, we may write $\mathbb{U} = \mathbb{U}^i \times \mathbb{U}^{-i}$ for any DM$^i$, and similarly for the sets $\bm{\Delta}$ and $\bm{\Pi}$. This allows us to re-write joint objects while isolating DM$^i$'s role: for instance, a joint action $\textbf{u} \in \mathbb{U}$ can be re-written as $\textbf{u} = (u^i, \textbf{u}^{-i})$ and a joint policy $\bm{\theta} \in \bm{\Delta}$ can be re-written as $\bm{\theta} = (\theta^i, \bm{\theta}^{-i})$.

A joint policy $\bm{\theta} \in \Delta$ induces a probability measure on sequences of states and joint actions, which we use in defining DM$^i$'s cost
\[
J^i_x ( \bm{\theta} ) := E^{\bm{\theta}} \bigg ( \sum_{t\in\mathbb{N}} (\beta^i)^t c^i ( x_t, \textbf{u}_t) \Big| x_0 = x   \bigg), \quad \forall x\in\mathbb{X}
\]
where $E^{\bm{\theta}}$ denotes that the stochastic process $\{ (x_t, \textbf{u}_t ) \}_{t \in \mathbb{N}}$ is determined by the policy $\bm{\theta}$. Then, each DM$^i$'s goal is to select a policy $\pi^i\in\Delta^i$ to minimize this cost.

\begin{definition}
A policy $\pi^{*i}\in\Delta^i$ is called a best reply to $\bm{\theta}^{-i}\in \bm{\Delta}^{-i}$ (for DM$^i$) if
\[
J_x^i (  \pi^{*i} , \bm{\theta}^{-i} ) = \min_{\pi^i \in \Delta^i} J_x^i ( \pi^i, \bm{\theta}^{-i}) , \quad \forall  x \in \mathbb{X}.
\]
Any best reply $\pi^{*i}\in\Delta^i$ to $\bm{\theta}^{-i}\in \bm{\Delta}^{-i}$ is called a strict best reply with respect to $(\pi^i,\bm{\theta}^{-i})$ if
\[
J_x^i (  \pi^{*i} , \bm{\theta}^{-i} ) < J_x^i ( \pi^i, \bm{\theta}^{-i}) , \quad \textrm{for some} \  x \in \mathbb{X}.
\]

\end{definition}
For any fixed $\bm{\theta}^{-i} \in \bm{\Delta}^{-i}$, DM$^i$ faces a stationary MDP; hence, DM$^i$ always has a deterministic best reply to any $\bm{\theta}^{-i} \in \bm{\Delta}^{-i}$. We denote the set of deterministic best replies by
\[
\BR^i(\bm{\theta}^{-i}) := \{ \pi^{*i} \in \Pi^i : \pi^{*i} \ \textrm{is a best reply to} \ \bm{\theta}^{-i} \}.
\]

\noindent We can describe the set $\BR^i(\bm{\theta}^{-i})$ using the optimal Q-factors for the MDP faced by DM$^i$ when playing against the policy $\bm{\theta}^{-i}$. The vector of optimal Q-factors for this environment is denoted $Q^{*i}_{\bm{\theta}^{-i}} \in \mathbb{R}^{\mathbb{X} \times \mathbb{U}^i}$. We include the policy $\bm{\theta}^{-i}$ in this notation as a reminder that the MDP and optimal Q-factors both depend on the policy used by all other players. Then, $\BR^i (\bm{\theta}^{-i})$ can be expressed as

\begin{align*}
%\BR^i (\bm{\theta}^{-i}) = \{ \pi^i \in \Pi^i : &Q^{*i}_{\bm{\theta}^{-i}} ( x, \pi^i(x)) \\
%						&= \min_{v^i \in \mathbb{U}^i} Q^{*i}_{\bm{\theta}^{-i}} ( x, v^i ), \forall x \in \mathbb{X} 
&\BR^i (\bm{\theta}^{-i}) \\
&= \{ \pi^i \in \Pi^i : Q^{*i}_{\bm{\theta}^{-i}} ( x, \pi^i(x)) = \min_{v^i \in \mathbb{U}^i} Q^{*i}_{\bm{\theta}^{-i}} ( x, v^i ), \forall x \in \mathbb{X} \}.
\end{align*}

\begin{definition}
A joint policy $\bm{\theta}^* \in \bm{\Delta} $ is called a (Markov perfect) equilibrium if $\theta^{*i}$ is a best reply to $\bm{\theta}^{*-i}$, for all $i$.
\end{definition}

We denote the set of all Markov perfect equilibrium policies by $\bm{\Delta}_{\rm eq}$ and we denote the set of stationary deterministic equilibrium policies by $\bm{\Pi}_{\rm eq} := \bm{\Delta}_{\rm eq} \cap \bm{\Pi}$. In any finite discounted stochastic game, the set $\bm{\Delta}_{\rm eq}$ is non-empty \cite{fudenberg1991game}. Note, however, that the set $\bm{\Pi}_{\rm eq}$ may be empty in general stochastic games.

\begin{definition}
A stochastic game is called a stochastic team (or simply a team) if there exists $c: \mathbb{X} \times \mathbb{U} \to \mathbb{R}$ and $\beta \in (0,1)$ such that
$$c^i = c, \ \beta^i=\beta, \quad \forall {\rm DM}^i .$$
%where $c$ and $\beta$ denote the common cost function and the discount factor of the team, respectively.
\end{definition}

\begin{definition}
A joint policy $\bm{\pi}^* \in \bm{\Pi}$ is called team-optimal if
\begin{equation}
\label{eq:teamopt}
J_x^i ( \bm{\pi}^* ) = \inf_{\bm{\pi} \in \bm{\Pi}} J_x^i ( \bm{\pi} ) \quad \forall i, x \in \mathbb{X}.
\end{equation}
\end{definition}

%\sout{We associate a team with a (centralized) stationary MDP with the joint action set $\mathbb{U}$. Any joint policy in $\Delta$ would be admissible for this MDP and thus the minimum cost for this MDP would be a lower bound on the cost of a team-optimal policy. Note that this centralized MDP would have a stationary deterministic optimal policy $\pi^*:\mathbb{X} \to \mathbb{U}$. Since $\pi^*\in\Pi$, $\pi^*$ would also be team-optimal. Therefore, in a team, the set of deterministic team-optimal policies $\Pi_{\rm opt} := \{ \pi \in \Pi : \pi \ \textrm{ is team-optimal}\}$ is always non-empty.}\footnote{Can we simply drop this whole section and replace it with the blue edit?}
We use $\bm{\Pi}_{\rm opt}$ to denote the set of team optimal policies, which are stationary deterministic policies by definition. It is easy to see that $\bm{\Pi}_{\rm opt}$ may be empty in a general stochastic game but that $\bm{\Pi}_{\rm opt}$ is non-empty in any stochastic team.

%\sout{The notion of team-optimality can be extended to a larger class of stochastic games than teams.}

\begin{definition}
\label{def:comint}
A stochastic game is called a common interest game if (i) $\bm{\Pi}_{\rm opt}$ is non-empty, and (ii) for any $\tilde{\bm{\pi}} \in \bm{\Pi} \setminus \bm{\Pi}_{\rm opt}$, we have
$$ \inf_{ \bm{\pi}  \in \bm{\Pi}}\sum_{x\in X}  J_x^i ( \bm{\pi}) < \sum_{x\in X} J_x^i ( \tilde{\bm{\pi}} ), \quad \forall {\rm DM}^i.$$

\end{definition}

This definition is consistent with the definition of a common interest game introduced in \cite{AUMANN19895} and used in other literature, e.g., \cite{TAKAHASHI2005231}.
Teams are a proper subclass of common interest games. The repeated game ($|\mathbb{X}|=1$) with the stage cost functions shown in Figure~\ref{fig1} is a common interest game for $a$, $b>0$ but not a team unless $a=b$ and $\beta^1=\beta^2$.

\begin{figure}[h]
\footnotesize
\centering
\begin{game}{3}{2}[$u_t^1$:][$u_t^2$:]
              & 1 & 2 \\
1     &$a,b$            & $a+1,b+1$ \\
2     &$a+1,b+1$            & $-a,-b$
\end{game}
\caption[]{Stage cost for a two-DM game where DM$^1$ (DM$^2$) chooses a row (a column) and its cost is the first (the second) entry in the chosen cell.}
\label{fig1}
\end{figure}

It is immediate that a team-optimal policy is an equilibrium; however, the converse need not be true. For an illustration of how poorly an equilibrium policy can perform with respect to team-optimality, consider again the repeated game presented in Figure~\ref{fig1} with $a=b>0$ and $\beta^1=\beta^2=\beta \in (0,1)$. Clearly, the joint policy $\bm{\pi}_{\rm sub}:=(1,1)$ is an equilibrium policy, and so is the team-optimal policy $\bm{\pi}^*:=(2,2)$. We have that $J^i( \bm{\pi}_{\rm sub} ) - J^i ( \bm{\pi}^*) =  \frac{2a}{1 - \beta}$, for each agent $i\in\{1,2\}$, which shows that the performance gap between an equilibrium policy and a team-optimal policy can be arbitrarily large. This provides the motivation for designing decentralized algorithms that allow agents to learn team-optimal policies, when they exist.

Our objective is the following: given a common interest game, we wish to provide each DM with a decentralized learning algorithm that does not use control sharing and that provably leads, in some appropriate sense, to a team optimal policy.

In \cite{AY2017}, we presented an algorithm that leads to equilibrium policies in weakly acyclic games, another class of games (different from common interest games) that generalizes teams. These algorithms instruct DMs to use the same stationary policy, called baseline policies, for large number of consecutive stages, the collection of which is called an exploration phase. At the end of an exploration phase, DMs update their baseline policies in a synchronized manner. In this way, the system is stationary for long enough for Q-learning to return meaningful Q-factors. The Q-factors acquired during an exploration phase are used to construct best replies; Q-factors are then reset for the next exploration phase. The DMs use inertial best-responding to update their baseline policies, and it is shown that this process leads to equilibrium policies in weakly acyclic games.

In the next section, we present a decentralized learning algorithm that leads to team-optimal policies, when they exist. The algorithm here uses the exploration phase technique from \cite{AY2017}, but modifies the baseline policy update in order to exploit the following structural result on Q-factors in teams and common interest games.

\begin{lemma} \label{QfactorsFact}
In a common interest game, for all $i$, $\bm{\pi}^* \in \bm{\Pi}_{\rm opt}$, $\tilde{\bm{\pi}} \in \bm{\Pi} \setminus \bm{\Pi}_{\rm opt}$, we have
\[
\sum_{x \in \mathbb{X}} Q^i_{\bm{\pi}^{*-i}} (x, \pi^{*i}(x) ) < \sum_{ x \in \mathbb{X}} Q^i_{\tilde{\bm{\pi}}^{-i}} (x ,\tilde{\pi}^i (x)).
\]
\end{lemma}

This fact provides for us an avenue for separating team-optimal policies from the other policies by focusing on Q-factors.

\begin{proof}

For all $i$, $\bm{\pi}^* \in \bm{\Pi}_{\rm opt}$,  $\tilde{\bm{\pi}} \in \bm{\Pi} \setminus \bm{\Pi}_{\rm opt}$, we have
\begin{align*}
\sum_{x \in \mathbb{X}} Q^i_{\bm{\pi}^{*-i}} ( x, \pi^{*i} ( x )) = & \sum_{x \in \mathbb{X}} J^i_x ( \bm{\pi}^* ) < \sum_{x \in \mathbb{X}} J^i_x (  \tilde{\bm{\pi}} )
\end{align*}
If $\tilde{\pi}^i\in \BR^i(\tilde{\bm{\pi}}^{-i})$, then $J_x^i ( \tilde{\bm{\pi}}) = Q^i_{\tilde{\bm{\pi}}^{-i}} ( x , \tilde{\pi}^i ( x ) )$; otherwise,
$$
\sum_{x \in \mathbb{X}} J^i_x ( \bm{\pi}^* ) \leq   \sum_{x \in \mathbb{X}} \min_{u^i\in\mathbb{U}^i} Q^i_{\tilde{\bm{\pi}}^{-i}} ( x , u^i)
                                         <  \sum_{x \in \mathbb{X}} Q^i_{\tilde{\bm{\pi}}^{-i}} ( x, \tilde{\pi}^i (x )) .
$$
\end{proof}

%%%%%%%%%%%%%%%%%%%%%%%%%%%%%%%%%%%%%%%%%%%%%%%%%%%%%%%%%%%%%%%%%%%%%%%%%%%%%%%%%%%%%%%%%%%%%%%%%%%%%%%%%%%%%%%%%%%%%%%%%%%%%%%%%%%%%%%%%%%%%%%%%%%%%%%%%%%%%%%%%%%%%%%%%%%%%%%%%%%%%%%%%%%%%%%%%%%%%%%%%%%%%%%%%%%%%%%%%%%%%%%%%%%%%%%%%%%%%%%%%%%%%%%%%%%%%%%%%%%%%%%%%%%%%%%%%%%%%%%%%%%%%%%%%%%%%%%%%%%%%%%%%%%%%%%%%%%%%%%%%%%%%%%%%%%%%%%%%%%%%%%%%%%%%%%%%%%%%%%%%%%%%%%%%%%%%%%%%%%%%%%%%%%%%%%%%%%%%%%%%%%%%%%%%%%%%%%%%%%%%%%%%%%%%%%%%%%%%%%%%%%%%%%%%%%%%%%%%%%%%%%%%%%%%%%%%%%%%%%%%%%%%%%%%%%%%%%%%%%%%%%%%

\section{Learning Team Optimality } \label{sec:mainAlgo}

In this section, we introduce a learning algorithm for achieving team optimality in teams and common interest games. To motivate our algorithm, we first study a time-homogenous Markov chain $\{ \bm{\pi}_k \}_{k \geq 0}$, taking values in the set of joint stationary deterministic policies $\bm{\Pi}$. The dynamics of this Markov chain will be determined by the Idealized Update Procedure (IUP), detailed in Algorithm~\ref{al:IUP}. While the IUP cannot be implemented in a stochastic common interest game under the information structure of interest, the resulting Markov chain will be used in approximation arguments in the proofs of our main results.

%Revised version, Feb 2021:
Under \textit{inertial best-responding} with inertia parameter $\lambda^i \in (0,1)$, at time $k \in \mathbb{N}$, DM$^i$ checks whether its current policy $\pi^i_k$ is a best-reply to the policy being used by other players, i.e. it checks if $\pi^i_k \in \BR^i (\bm{\pi}^{-i}_k)$; if it is, then $\pi^i_{k+1} = \pi^i_k$. Otherwise DM$^i$ is not best-replying and selects $$\pi^i_{k+1} \sim (1-\lambda^i) \uniform ( \BR^i ( \bm{\pi}^{-i}_k )) + \lambda^i \mathbb{I}_{\pi^i_k},$$

that is, switches to a random best-reply with probability $1 - \lambda^i$ or is inert (does not change away from $\pi^i_k$) with probability $\lambda^i$. Including inertia in one's policy update can be used to avoid cycling in best-reply dynamics. For example, in the game in Figure \ref{fig1}, if play starts at either joint policy $(1,2)$ or at $(2,1)$ and both players switch to a best-reply at each step, the joint policy will cycle between $(1,2)$ and $(2,1)$ perpetually. Such cycling can be avoided by using explicit coordination mechanisms for determining which DM should change its policy and at what time, but such mechanisms may not be feasible in decentralized settings. Simple decentralized mechanisms such as inertia can been used with the same effect \cite{marden2012revisiting,marden2009payoff}.

The condition-dependent nature of inertial best-responding can be captured using a stochastic kernel $R^{i, \lambda^i} \in \mathcal{P} ( \Pi^i | \Pi^i \times 2^{\Pi^i} )$, where $R^{i, \lambda^i}$ selects a successor policy randomly, conditioning on the current policy and the current (perhaps estimated) best-reply set. To allow for uncertainty of $\BR^i ( \bm{\pi}^{-i}_k)$, we define $R^{i, \lambda^i}$ as follows:

\begin{equation}
\label{eq:Ri}
R^{i,\lambda^i}( \tilde{\pi}^i | \pi^i,B^i) :=
\left\{\begin{array}{cl}
1, &  \textrm{if} \ \pi^i\in B^i, \ \tilde{\pi}^i=\pi^i \\
\lambda^i, &  \textrm{if} \ \pi^i\not\in B^i, \ \tilde{\pi}^i=\pi^i \\
\frac{1-\lambda^i}{|B^i|}, &  \textrm{if} \ \pi^i\not\in B^i, \ \tilde{\pi}^i\in B^i \\
0, & \textrm{otherwise}
\end{array}\right. ,
\end{equation}
for any $\pi^i \in \Pi^i, B^i \in 2^{\Pi^i}$ and $\tilde{\pi}^i \in \Pi^i$.

Note that selecting $\pi^i_{k+1} \sim R^{i,\lambda^i}(\cdot | \pi^i,\BR^i(\bm{\pi}_k^{-i}))$ is equivalent to selecting $\pi^i_{k+1}$ according to inertial best-responding with parameter $\lambda^i$.

Under the \textit{Idealized Update Procedure} (IUP), presented in Algorithm \ref{al:IUP}, DM$^i$ chooses $\pi^i_{k+1}$ according to a mixture of uniform random experimentation and inertial best-responding when the joint policy is team optimal, i.e. $\bm{\pi}_k \in \bm{\Pi}_{\rm opt}$. %In contrast, DM$^i$ has more flexibility in how it updates its policy when $\bm{\pi}_k \notin \bm{\Pi}_{\rm opt}$; in this case, DM$^i$ uses a mixture of uniform random experimenting and a player selected stochastic kernel $h^i  \in \mathcal{P} (\Pi^i | \Pi^i \times 2^{\Pi^i} )$ to choose $\pi^i_{k+1}$. (In particular, the guarantee of Lemma \ref{IdealizedProcessLemma}, on attaining team-optimality in common interest games, hold for with arbitrary $\{h^i\}_{i=1}^N$.)
When $\bm{\pi}_k \notin \bm{\Pi}_{\rm opt}$, DM$^i$ uses a mixture of uniform random experimenting and a player selected stochastic kernel $h^i  \in \mathcal{P} (\Pi^i | \Pi^i \times 2^{\Pi^i} )$ to choose $\pi^i_{k+1}$.

\begin{algorithm} \label{al:IUP}
\caption{Idealized Update Procedure (IUP) for DM$^i$}

\SetAlgoLined
\DontPrintSemicolon
\SetKw{Receive}{Receive}
\SetKw{parameters}{Set Parameters}
\SetKw{initialize}{Initialize}
%\SetKwFor{For}{for}{}{endfor}
\SetKwBlock{For}{for}{end} 

\parameters \;
$\lambda^i\in [0,1]$: inertia probability \;
$h^i \in \mathcal{P} ( \Pi^i | \Pi^i \times 2^{\Pi^i} )$, a policy update kernel  \; 
$\gamma^i , \kappa^i \in (0,1)$: exploration probabilities \; 
\BlankLine

\For($k \geq 0$ ){ %i  had to put { and } around ) because otherwise it misunderstands it as ending the for-loops condition. 

	\lIf(){ $\bm{\pi}_k \in \bm{\Pi}_{\rm opt}$ }{ $\pi^i_{k+1} \sim (1 - \gamma^i) R^{i, \lambda^i}( \cdot | \pi_k^i,\BR^i(\bm{\pi}_k^{-i})) + \gamma^i \uniform (\Pi^i)$ }
	\Else( {($\bm{\pi}_k \notin \bm{\Pi}_{\rm opt} $) }  ){
	$\pi^i_{k+1} \sim (1 - \kappa^i ) h^i ( \cdot | \pi^i_k , \BR^i ( \bm{\pi}^{-i}_k) + \kappa^i \uniform ( \Pi^i ) $
	}
	
}

\end{algorithm}

We will require that DM$^i$ randomly explores $\Pi^i$ more when the joint policy is not team optimal, i.e. $\kappa^i \gg \gamma^i$. Qualitatively, this results in shifting away from suboptimal joint policies more quickly than team optimal policies, and as a result the process spends a large fraction of time in $\bm{\Pi}_{\rm opt}$.  We formalize this intuition below, and note that the guarantee of Lemma \ref{IdealizedProcessLemma}, on attaining team-optimality in common interest games, holds for arbitrary $\{h^i\}_{i=1}^N$. That is, DM$^i$ has some flexibility in how it updates its policies when not experimenting and when the current joint policy is not team optimal.

\begin{lemma} \label{IdealizedProcessLemma}
Consider a common interest game, and suppose each DM$^i$ updates its policies according to the IUP in Algorithm \ref{al:IUP}. Let $A_{\bm{\gamma},\bm{\kappa}, \textbf{h}}$ denote the matrix of the transition probabilities for the induced time-homogenous Markov chain on $\bm{\Pi}$, where $\bm{\gamma}:=\{\gamma^i\}_{i=1}^N$, $\bm{\kappa}:=\{\kappa^i\}_{i=1}^N$, $\textbf{h} = \{ h^i \}_{i = 1}^N$.
We denote the associated unique stationary distribution by $\mu^*_{\bm{\gamma}, \bm{\kappa} , \textbf{h}}$. For any $\epsilon\in(0,1)$, $\bm{\kappa}\in(0,1)^N$, there exists $\bar{\gamma}_{\epsilon}(\bm{\kappa}) > 0$ such that if $\gamma^i \in (0 , \bar{\gamma}_{\epsilon}(\bm{\kappa}) )$ for all $i$, then
\begin{equation}
\label{eq:mu*}
\mu^*_{\bm{\gamma}, \bm{\kappa}, \textbf{h}} ( \bm{\Pi}_{\rm opt} ) \geq 1 - \epsilon/2.
\end{equation}
Moreover, for all $\mu_0\in\mathcal{P}(\bm{\Pi})$, we have $$\lim_{n \to \infty} \mu_0 A_{\bm{\gamma,\kappa}, \textbf{h}}^n = \mu^*_{\bm{\gamma,\kappa, h}}.$$
\end{lemma}

\begin{proof}
Since $\gamma^i, \kappa^i > 0$ for all $i$, the induced Markov chain is irreducible, hence there exists unique $\mu_{\bm{\gamma,\kappa, h}}^*$ such that $\mu_{\bm{\gamma,\kappa, h}}^* = \mu_{\bm{\gamma,\kappa, h}}^* A_{\bm{\gamma,\kappa, h}}$. We have 
\begin{flalign*}
\sum_{ \bm{\pi}^*  \in \bm{\Pi}_{\rm opt}} \mu_{\bm{\gamma,\kappa, h}}^* ( \bm{\pi}^* ) \\
=   \sum_{\bm{\pi}^* \in \bm{\Pi}_{\rm opt}} \sum_{\bm{\pi} \in \bm{\Pi}_{\rm opt}} \mu_{\bm{\gamma,\kappa, h}}^* ( \bm{\pi} )  A_{\bm{\gamma,\kappa, h}} ( \bm{\pi}, \bm{\pi}^* ) \\
\quad +  \sum_{\bm{\pi}^* \in \bm{\Pi}_{\rm opt}} \sum_{\bm{\pi} \notin \bm{\Pi}_{\rm opt}} \mu_{\bm{\gamma,\kappa, h}}^* ( \bm{\pi})  A_{\bm{\gamma,\kappa, h}} ( \bm{\pi}, \bm{\pi}^* )  \\
\geq    \sum_{\bm{\pi} \in \bm{\Pi}_{\rm opt}} \mu_{\bm{\gamma,\kappa, h}}^* ( \bm{\pi} ) \prod_i (1 - \gamma^i) \\
\quad+ \sum_{\bm{\pi} \notin \bm{\Pi}_{\rm opt}} \mu_{\bm{\gamma,\kappa, h}}^* ( \bm{\pi} ) \prod_i (\kappa^i/|\Pi^i|)
\end{flalign*}
%Feb 20th, 2021
%\begin{align*}
%\sum_{ \bm{\pi}^*  \in \bm{\Pi}_{\rm opt}} &\mu_{\bm{\gamma,\kappa, h}}^* ( \bm{\pi}^* ) \\
%&=   \sum_{\bm{\pi}^* \in \bm{\Pi}_{\rm opt}} \sum_{\bm{\pi} \in \bm{\Pi}_{\rm opt}} \mu_{\bm{\gamma,\kappa, h}}^* ( \bm{\pi} )  A_{\bm{\gamma,\kappa, h}} ( \bm{\pi}, \bm{\pi}^* ) \\
%&\quad +  \sum_{\bm{\pi}^* \in \bm{\Pi}_{\rm opt}} \sum_{\bm{\pi} \notin \bm{\Pi}_{\rm opt}} \mu_{\bm{\gamma,\kappa, h}}^* ( \bm{\pi})  A_{\bm{\gamma,\kappa, h}} ( \bm{\pi}, \bm{\pi}^* )  \\
%&\geq    \sum_{\bm{\pi} \in \bm{\Pi}_{\rm opt}} \mu_{\bm{\gamma,\kappa, h}}^* ( \bm{\pi} ) \prod_i (1 - \gamma^i) \\
%&\quad+ \sum_{\bm{\pi} \notin \bm{\Pi}_{\rm opt}} \mu_{\bm{\gamma,\kappa, h}}^* ( \bm{\pi} ) \prod_i (\kappa^i/|\Pi^i|)
%\end{align*}
%Old version, after boldfacing 
%\begin{align*}
%\sum_{ \bm{\pi}^*  \in \bm{\Pi}_{\rm opt}} \mu_{\bm{\gamma,\kappa, h}}^* ( \bm{\pi}^* ) =    \sum_{\bm{\pi}^* \in \bm{\Pi}_{\rm opt}} \sum_{\bm{\pi} \in \bm{\Pi}_{\rm opt}} \mu_{\bm{\gamma,\kappa, h}}^* ( \bm{\pi} )  A_{\bm{\gamma,\kappa, h}} ( \bm{\pi}, \bm{\pi}^* ) \\
% +  \sum_{\bm{\pi}^* \in \bm{\Pi}_{\rm opt}} \sum_{\bm{\pi} \notin \bm{\Pi}_{\rm opt}} \mu_{\bm{\gamma,\kappa, h}}^* ( \bm{\pi})  A_{\bm{\gamma,\kappa, h}} ( \bm{\pi}, \bm{\pi}^* )  \\
%\geq    \sum_{\bm{\pi} \in \bm{\Pi}_{\rm opt}} \mu_{\bm{\gamma,\kappa, h}}^* ( \bm{\pi} ) \prod_i (1 - \gamma^i) \\
%+ \sum_{\bm{\pi} \notin \bm{\Pi}_{\rm opt}} \mu_{\bm{\gamma,\kappa, h}}^* ( \bm{\pi} ) \prod_i (\kappa^i/|\Pi^i|)
%\end{align*}
This leads to
\[
\sum_{ \bm{\pi}^*  \in \bm{\Pi}_{\rm opt}} \mu_{\bm{\gamma,\kappa, h}}^* ( \bm{\pi}^* ) \geq 1 - \frac{ \sum_i \gamma^i}{\sum_i \gamma^i + \prod_i (\kappa^i/|\Pi^i|)}
\]
which implies (\ref{eq:mu*}). The last part follows from the aperiodicity of the Markov chain.
\end{proof}

Lemma~\ref{IdealizedProcessLemma} shows that if DMs follow the IUP, then they would choose a team-optimal policy in the long run with arbitrarily high probability provided the experimentation probabilities of $\bm{\gamma}$ are positive but sufficiently small relative to $\bm{\kappa}$. 

It is clear that the IUP cannot be directly implemented in our study of decentralized, online teams. The first issue relates to decentralization: DM$^i$ cannot observe the policy $\bm{\pi}^{-i}_k$. The second issue relates to the online nature of the problem: even if $\bm{\pi}^{-i}_k$ were known, DM$^i$ may not know its best-reply set or the set of team optimal policies. Nevertheless, the IUP motivates our decentralized learning algorithm, Algorithm~\ref{al:main}, which can be viewed as a two timescale approximation of the IUP. We expand on this point below, after presenting the main result of this section.

\begin{algorithm}[h] \label{al:main}
\SetAlgoLined
\DontPrintSemicolon
\SetKw{Receive}{Receive}
\SetKw{parameters}{Set Parameters}
\SetKw{initialize}{Initialize}
%\SetKwFor{For}{for}{}{endfor}
\SetKwBlock{For}{for}{end} 

\parameters \;
\Indp 
$\mathbb{Q}^i \subset \mathbb{R}^{ \mathbb{X}\times\mathbb{U}^i }$: a compact set \;
$\{ T_k \}_{k \geq 0}$: a sequence in $\mathbb{N}_+$ of exploration phase lengths (common to all DMs) \;
\Indp Set $t_0 = 0$ and $t_{k+1} = t_k + T_k$ for all $k \geq 0.$ \;
\Indm $\rho^i\in(0,1)$: action experimentation probability \;
$\gamma^i, \kappa^i \in ( 0, 1 )$: policy experimentation probabilities \; 
$\lambda^i\in [0,1]$: probability of inertia when updating baseline policy \;
$\delta^i > 0$: tolerance for sub-optimality when constructing best-reply sets \;
$d^i > 0$: a tolerance for sub-optimality when setting the aspiration level \;
$W^i \in \mathbb{N}_+$: a window for setting aspiration levels \; 
$h^i \in \mathcal{P} ( \Pi^i | \Pi^i \times 2^{\Pi^i} )$, a policy update kernel  \; 
$\{\alpha_{n}^{i}\}_{n\geq0}$: step sizes such that $\alpha_{n}^{i}\in[0,1]$, $\sum_n\alpha_n^{i} = \infty$, $\sum_{n} \big(\alpha_n^{i}\big)^2 < \infty$ \;
\Indm 
\BlankLine

\initialize (arbitrary) baseline policy $\pi_0^i \in \Pi^i$, $Q_0^i\in\mathbb{Q}^i$ \\
\Receive $x_0$ \\

\For($k \geq 0$ ($k^{th}$ exploration phase{)} ){ %i  had to put { and } around ) because otherwise it misunderstands it as ending the for-loops condition. 
	\For( $t = t_k, t_k +1, \dots, t_{k+1} - 1$  \tcp*[f]{ Learn best-replies for $k^{th}$ EP} ){
	Select  $u^i_t \sim (1-\rho^i) \mathbb{I}_{\pi^i_k ( x_t ) } + \rho^i \uniform ( \mathbb{U}^i)$				\;
	%$u^i_t = 	\begin{cases} 
	%			\pi^i_k ( x_t ), \quad &{\rm w.p. } \quad 1 - \rho^i \\ 
	%			u^i \sim \uniform ( \mathbb{U}^i ), \quad &{\rm w.p. } \quad \rho^i 
	%		\end{cases}$ \;
	\Receive cost $c^i( x_t, u^i_t, \textbf{u}^{-i}_t)$ \;
	\Receive state $x_{t+1} \sim P ( \cdot | x_t, \textbf{u}_t) $ \; 
	Set $n_t^i =$ number of visits to $(x_t,u_t^i)$ in $[t_k, t]$ \;
	$Q_{t+1}^i(x_t,u_t^i)  =   (1-\alpha_{n_t^i}^{i})Q_t^i(x_t,u_t^i) + \alpha_{n_t^i}^{i}  \big[ c^i(x_t,u_t^i, \textbf{u}_t^{-i}) +  \beta^i \min_{v^i} Q_t^i(x_{t+1},v^i) \big]$  \; 
	 $Q_{t+1}^i(x,u^i) =   Q_t^i(x,u^i)$,   $\forall (x,u^i)\not=(x_t,u_t^i)$\; 
	 }
	${\rm BR}^i_k = \{ \pi^i \in \Pi^i : Q^i_{t^i_{k+1}} ( x, \pi^i ( x) ) \leq \min_{v^i} Q^i_{t^i_{k+1}} (x, v^i) + \delta^i, \forall x \in \mathbb{X} \}$ \;
	$S^i_k = \sum_{x \in \mathbb{X}} Q^i_{t^i_{k+1}} (x, \pi^i_k ( x))$ \;
	$\Lambda^i_k = \min \{ S^i_{k-1}, \dots, S^i_{(k - W^i )^+} \} + d^i$ \;

	%If suite template from the asynchronous paper: 	
%	\If(\tcp*[f]{Baseline policy update}){ $\pi^i_k \in \BR^i_k$ }{ $\pi^i_{k+1} \leftarrow \pi^i_k$ }
%	\Else(){
%		$\pi^i_{k+1} \leftarrow 	\begin{cases} 	
%								\pi^i_k , &{\rm w.p. } \quad \lambda^i \\
%								\pi^i \sim \uniform ( \BR^i_k ), &{\rm w.p.} \quad 1 - \lambda^i 	
%							\end{cases}$
%		\BlankLine
%	}

	\lIf(){ $S^i_k \leq \Lambda^i_k$ }{ $\pi^i_{k+1} \sim (1-\gamma^i) R^{i, \lambda^i} ( \cdot | \pi^i_k , \BR^i_k ) + \gamma^i \uniform ( \Pi^i )$ }
	\Else({(} $S^i_k > \Lambda^i_k$, not achieving aspiration{)}){
	$\pi^i_{k+1} \sim (1 - \kappa^i ) h^i ( \cdot | \pi^i_k , \BR^i_k ) + \kappa^i \uniform ( \Pi^i ) $
	}

	Reset $Q_{t^i_{k+1}}^i$ to any $Q^i\in\mathbb{Q}^i$ (e.g., project onto $\mathbb{Q}^i $) \;
}

\caption{Independent Team Q-Learning for DM$^i$}
\end{algorithm}

We emphasize that Algorithm \ref{al:main} is decentralized in the sense that it can be implemented by ``independent learners," in the terminology of \cite{matignon2012Survey,zhang2019survey}. That is, each DM$^i$ can run a separate copy of this algorithm without reference to the joint actions or policies of the remaining players. We recall that each DM$^i$'s interaction with its environment at any time $t$ consists of sending its control decision $u_t^i$ and receiving its cost realization $c^i(x_t,u_t^1,\dots,u_t^N)$ as well as the next state $x_{t+1}$ without observing any information about the other DMs, in particular, without observing the control decisions $\textbf{u}_t^{-i}$ of the other DMs. In fact, each DM$^i$ need not even be aware of the presence of the other DMs or the fact it is engaged in learning in a multi-player game. Simply, each DM is running a single-agent algorithm similar to standard Q-learning (that is re-initialized after its baseline policy is updated at the end of each exploration phase). As such, all quantities computed by DM$^i$'s copy of Algorithm~\ref{al:main} are indexed by $i$. These remarks also apply verbatim to Algorithm~\ref{al:co-main} introduced in Section~\ref{sec:constasp}.

\begin{assumption} \label{noTransientStates}
For all $x, x' \in \mathbb{X}$, there exists $H \in\mathbb{N}$ and $\tilde{\textbf{u}}_0, \dots, \tilde{\textbf{u}}_H\in\mathbb{U}$ such that
\[
\P (x_{H+1}= x' \vert x_0 = x, \textbf{u}_j = \tilde{\textbf{u}}_j, \forall j \in \{0, 1, \dots, H\} ) > 0.
\]
\end{assumption}

%Assumption \ref{noTransientStates} is common to reinforcement learning methods. If it does not hold, then there exist transient states that will be visited a number of times and then never revisited. The play in these states affects the long run discounted cost, but there will be no opportunity for subsequent experimentation, and even cleverly designed learning algorithms will fail to reliably find optimal policies in such settings.

%\bmy{Commentary on assumption 1: in games where this does not hold, there may be a state that is transient. If this is the case, then Q-learning will not find a globally optimal behaviour, since the action in that state influences the long run payoff (as we consider discounted and not average cost), but cannot be simulated infinitely often.}

\begin{assumption} \label{suffSmallAssumption}
Assume, for all $i$, $\delta^i \in(0, \bar{\delta})$, $d^i \in (0, \bar{d})$, $\rho^i \in (0,\bar{\rho})$, where $\bar{\delta}$, $\bar{d}$, $\bar{\rho}$ are constants defined in Appendix~A that depend only on the game.
\end{assumption}

\begin{theorem} \label{theorem1}
Consider a common interest game in which each DM$^i$ uses Algorithm~\ref{al:main}, and let Assumptions~\ref{noTransientStates}-\ref{suffSmallAssumption} hold. For any $\epsilon > 0$, there exist
\begin{gather*}
\bar{\gamma}_{\epsilon}(\bm{\kappa}) \in(0,1), \quad
\bar{W}_{\epsilon}(\bm{\gamma,\kappa}) \in \mathbb{N}_+, \quad
\bar{T}_{\epsilon} (\bm{\gamma,\kappa},W_\max) \in \mathbb{N}_+
\end{gather*}
where $W_\max:=\max_i W^i$ such that if, for all $i$, $k\in\mathbb{N}$
\begin{gather*}
\gamma^i \in (0, \bar{\gamma}_{\epsilon}(\bm{\kappa})), \quad W^i \geq \bar{W}_{\epsilon}(\bm{\gamma,\kappa}), \quad T_k  \geq \bar{T}_{\epsilon} (\bm{\gamma,\kappa},W_\max)
\end{gather*}
then
\[
\liminf_{k\in\mathbb{N}} \P ( \bm{\pi}_k \in \bm{\Pi}_{\rm opt}) \geq 1 - \epsilon
\]
\end{theorem}

%\bmy{The parenthetical parameters are to show the dependencies. Also, as noted in the appendix, the proof of Theorem 1 shows that instead of a liminf statement, I can even say $\exists M \in \mathbb{N} : k \geq M$ implies $P (\pi_k \in \Pi^{OPT} ) \geq 1 - \epsilon$. }

%%%%%%%%%%%%%%%%%%%%%%%%%%%%%%%%%%%%%%%%%%%%%%%%%%%%%%%%%%%%%%%%%%%%%%%%%%%%%%%%%%%%%%%%%%%%%%%%%%%%%

%%%%%%%%%%%%%%%%%%%%%%%%%%%%%%%%%%%%%%%%%%%%

%\input{theoremA}

\begin{proof} See Appendix~A. \end{proof}

\subsection*{Discussion}

Algorithm \ref{al:main} can be viewed as a two timescale\footnote{In two timescale algorithms in the literature (e.g.  \cite{borkar1997stochastic}), both the Q-factors and the policies would be updated incrementally at each time $t=1,2,\dots$. The step size sequences for Q-learning and policy updating would be selected so that policies are effectively fixed while Q-factors are learned. In our algorithms, the policies are updated without using any step sizes but only only at $t=t_1-1,t_2-1,\dots$ whereas the Q-factors are updated at each time $t=1,2,\dots$ using step sizes that are re-initialized at $t=t_1-1,t_2-1,\dots$ and are reduced during $t\in[t_k,t_{k+1}-1)$ at a rate satisfying the assumptions of the standard (i.e., one time scale) stochastic approximation theory. }
 approximation to the IUP in Algorithm \ref{al:IUP}. The faster timescale is that where time is indexed by the stage games, comprising lines 16-23 of Algorithm~\ref{al:main}. The selection of actions, observation of costs and state transitions, and Q-factor updates all occur on this faster timescale. In contrast, the slower timescale is where time is indexed by the exploration phase. Decisions on the slower timescale involve processing learned Q-factors to estimate one's best-reply set (line 24), computing a ``cost score'' and comparing it to historical cost scores (lines 25-27), and updating one's baseline policy (lines 27-31). 

 Crucially, the baseline policies are fixed within an exploration phase and only change between exploration phases. This means that for any $k \geq 0$, from the point of view of any DM$^i$, the environment is stationary within the $k^{th}$ EP and equivalent to an MDP determined by $\bm{\pi}^{-i}_k$. It was shown in \cite{AY2017} that under certain conditions---satisfied here by Assumptions \ref{noTransientStates} and \ref{suffSmallAssumption}---that Q-learning within an EP leads to informative Q-factors that can be used, among other things, to recover one's best-reply set $\BR^i ( \bm{\pi}^{-i}_k )$. 
 
The analogy between Algorithm~\ref{al:main} and the IUP can be seen by comparing the if-suite (lines 6-9) in Algorithm~\ref{al:IUP} with its counterpart (lines 27-30) in Algorithm~\ref{al:main}. The unobservable condition $\bm{\pi}_k \in \bm{\Pi}_{\rm opt}$ of the IUP has been replaced by a surrogate condition $S^i_k \leq \Lambda^i_k$. Here, $S_k^i$ is a ``cost score," which aggregates DM$^i$'s policy performance across all states for the $k^{th}$ exploration phase, and $\Lambda_k^i$ is a measure of DM$^i$'s best performance during the preceding $W^i$ exploration phases. Importantly, the condition $S^i_k \leq \Lambda^i_k$ can be verified by independent learners.

Algorithm~\ref{al:main} is in the spirit of aspiration learning algorithms \cite{chasparis2013aspiration}, where $\Lambda_k^i$ plays the role of DM$^i$'s aspiration level, a scalar quantity against which DM$^i$ compares the performance of its policy $\pi_k^i$ during the $k^{th}$ exploration phase. Each DM$^i$ aspires to perform at least as well as its aspiration level, which is updated at the end of each exploration phase and may be thought of as a maximum tolerable cost, i.e., if the baseline policy yields higher cost, then it is viewed as unsatisfactory.

Unlike the aspiration learning methods in the literature, which focus on repeated games with no state dynamics and players with no look ahead, Algorithm~\ref{al:main} is designed for stochastic dynamic games with nontrivial state dynamics and far-sighted players. Due to the long-run cost considerations in dynamic stochastic games, evaluating of the cost of a policy is a slow and noisy process, which leads to additional difficulties in setting the aspiration levels.

In light of Lemma \ref{QfactorsFact}, a viable approach is to use the learned Q-factors to produce cost scores and to set the aspiration levels to the minimum cost score over some window of the past. However, scores obtained from the (random) Q-factors are noisy estimates of the scores corresponding to the true cost of the policies. In particular, setting the aspiration levels to the minimum of the cost scores over the entire past based on the learned Q-factors can result in unattainable aspiration levels. Hence, to mitigate the effects of the noise present in the learned Q-factors, we set the aspiration levels of each DM$^i$ to the minimum cost score obtained over a finite window of the most recent past within some tolerance. This allows DMs to discard unattainable cost scores in finite time.

Another aspect of Algorithm~\ref{al:main} is the persistent experimentation in the policy space. Experimentation when DMs feel that they meet their aspiration levels ($S_k^i\leq\Lambda_k^i)$ is required to prevent DMs settling in a policy that is not team-optimal. This is due to the finite window approach used for setting the aspiration levels and the possibility of setting suboptimal aspiration levels. Experimentation when ($S_k^i>\Lambda_k^i)$ is also necessary to aid DMs in searching for team-optimal policies.

Finally, we note that the set of approximate best responses $\BR_k^i$ computed by each DM$^i$ within each exploration phase $k$ is a subset of $\Pi^i$, the set of stationary and deterministic policies of DM$^i$. Therefore, $|\BR_k^i|\leq|\Pi^i|=|\mathbb{U}^i|^{|\mathbb{X}|}$. We note that  $\BR_k^i$ is computed via the Q-factors $Q^{i}_{t_k+1} \in \mathbb{R}^{\mathbb{X} \times \mathbb{U}^i}$, which is of size $|\mathbb{X}| |\mathbb{U}^i|$.

%Remove newpage command once the prev. section is formatted right. 
%\newpage
\section{Beyond Team Optimality: Application to Weakly Acyclic Games} \label{sec:weaklyAcyclic}

In this section, we consider a special case of Algorithm~\ref{al:main} that has desirable convergence properties in weakly acyclic games, in addition to providing team-optimality in the sense of Theorem~\ref{theorem1}.

\begin{definition}
A (possibly finite) sequence $\bm{\pi}_0 , \bm{\pi}_1 , \dots$ in $\bm{\Pi}$ is called a multi-DM strict best reply path if, for each $k$, $\bm{\pi}_k$ and $\bm{\pi}_{k+1}$ differ for at least one DM and, for each deviating DM$^i$, $\bm{\pi}_{k+1}^i$ is a strict best reply with respect to $\bm{\pi}_k$.
\end{definition}

\begin{definition} A stochastic game is called weakly acyclic under multi-DM strict best replies (or simply weakly acyclic) if there is a multi-DM strict best reply path starting from each deterministic joint policy and ending at a deterministic equilibrium policy.
\end{definition}

The notion of weak acyclicity used here is with respect to stationary deterministic policies for stochastic games, and generalizes the notion of weak acyclicity introduced in \cite{young-book} for single-stage games. All teams are weakly acyclic; however, a common interest game need not be. See \cite{fabrikant2010structure} for other examples of single-stage weakly acyclic games.

In weakly acyclic games, the inertial best reply dynamics \cite{AY2017} lead to equilibrium policies. If the policy update functions satisfy $h^i = R^{i, \lambda^i}$ for each DM$^i$, the IUP introduced in the previous section can be regarded as a perturbed inertial best reply dynamics, where $\{\bm{\pi}_k\in \bm{\Pi}_{\rm opt}\}$ can be replaced with any arbitrary event if the game is not a common interest game, provided the induced Markov chain is time-homogenous.

\begin{assumption}
\label{as:iup2}
For every DM$^i$, $h^i = R^{i,\lambda^i}$, where $\lambda^i \in (0,1)$.
\end{assumption}

Under Assumption~\ref{as:iup2}, each DM$^i$ always best replies with inertia when not experimenting.

\begin{lemma} \label{IdealizedProcessLemma2}
Consider a weakly acyclic game. Suppose that each DM$^i$ updates its policy according to the IUP of Algorithm \ref{al:IUP}, and let Assumption~\ref{as:iup2} hold. Let $A_{\bm{\gamma, \kappa}}$ denote the matrix of the transition probabilities for the induced time homogenous Markov chain on $\bm{\Pi}$.
Denote the unique stationary distribution associated to this Markov chain by $\mu^*_{\bm{\gamma, \kappa}}$. For any $\epsilon > 0$, there exists $\bar{\kappa}_{\epsilon} \in (0,1)$  such that $\max\{\gamma^i,\kappa^i\} \in (0 , \bar{\kappa}_{\epsilon} )$, for all $i$, implies
\begin{equation}
\nonumber
\mu^*_{\bm{\gamma, \kappa}} ( \bm{\Pi}_{\rm eq} ) \geq 1 - \epsilon/4.
\end{equation}
Moreover, uniformly over all such $\bm{\gamma, \kappa}$, there exists $\bar{m}\in\mathbb{N}$ such that
$$\inf_{m\geq\bar{m}, \mu_0\in\mathcal{P}(  \bm{\Pi} )} (\mu_0 A_{\bm{\gamma, \kappa}}^m)(\bm{\Pi}_{\rm eq}) \geq 1-\epsilon/2.$$
\end{lemma}

\begin{proof}
For all $\bm{\pi}^* \in \bm{\Pi}_{\rm eq}$,
\begin{equation}
\label{eq:brwiel1}
A_{\bm{\gamma, \kappa}} (\bm{\pi}^*,\bm{\pi}^* )  \geq \prod_i (1 - \max\{\gamma^i,\kappa^i\})
\end{equation}
Let $L_{\bm{\pi}}<|\bm{\Pi}|$ be the length of a multi-DM strict best reply path of minimal length from $\bm{\pi} \in \bm{\Pi}\setminus \bm{\Pi}_{\rm eq}$ to some $\tilde{\bm{\pi}}\in \bm{\Pi}_{\rm eq}$, and $L:=\max_{\bm{\pi} \in \bm{\Pi} \setminus \bm{\Pi}_{\rm eq}} L_{\bm{\pi}}$.
For any $\bm{\pi} \not\in  \bm{\Pi}_{\rm eq}$, consider a path $\bm{\pi} = \bm{\pi}_0, \bm{\pi}_1,\dots, \bm{\pi}_L$ where $\bm{\pi}_0, \bm{\pi}_1,\dots, \bm{\pi}_{L_{\bm{\pi}}}$ is a multi-DM strict best reply path and $\bm{\pi}_{L_{\bm{\pi}}}=\cdots=\bm{\pi}_L = \tilde{\bm{\pi}} \in \bm{\Pi}_{\rm eq}$. In each transition $\bm{\pi}_k\rightarrow \bm{\pi}_{k+1}$, some DMs switch to one of their strict best replies and the others stay put. Therefore, from any $\bm{\pi} \not\in \bm{\Pi}_{\rm eq}$, the IUP with $\bm{\gamma} = \bm{\kappa}  \equiv 0$ generates such a path $\bm{\pi}_0, \bm{\pi}_1,\dots, \bm{\pi}_L$ with probability at least $p_{\min}:=\prod_{i=1}^N\min\{\lambda^i,(1-\lambda^i)/|\Pi^i|\}^L\in(0,1)$. By taking $\gamma^i>0$, $\kappa^i>0$ into account, this leads to
\begin{equation}
\label{eq:brwiel2}
\sum_{\tilde{\bm{\pi}}\in\bm{\Pi}_{\rm eq}} (A_{\bm{\gamma, \kappa}})^L (\bm{\pi},\tilde{\bm{\pi}})  \geq p_\min \prod_i (1 - \max\{\gamma^i,\kappa^i\})^L
\end{equation}
for all $\bm{\pi} \in \bm{\Pi} \setminus \bm{\Pi}_{\rm eq}$. Writing $A = A_{\bm{\gamma, \kappa}}$, from (\ref{eq:brwiel1})-(\ref{eq:brwiel2}), we have, for all $k\in\mathbb{N}$,
\begin{align*}
(\mu_0 A^{k+L})( \bm{\Pi} \setminus \bm{\Pi}_{\rm eq} ) \leq & L\sum_i \max\{\gamma^i,\kappa^i\} \\
& + (\mu_0 A^k)( \bm{\Pi} \setminus \bm{\Pi}_{\rm eq} )(1 -  p_\min).
\end{align*}
This leads to, for all $j$, $k\in\mathbb{N}$,
$$(\mu_0 A^{k+jL})(\bm{\Pi} \setminus \bm{\Pi}_{\rm eq} ) \leq  L\sum_i \max\{\gamma^i,\kappa^i\}/p_\min + (1 -  p_\min)^j.$$
Since $|1-p_\min|<1$, the desired result follows.
\end{proof}

For small experimentation probabilities, the IUP under Assumption~\ref{as:iup2} leads to equilibrium policies in the long run. We will use this to show that Algorithm~\ref{al:main} under Assumptions~\ref{noTransientStates}-\ref{as:iup2} has the same long run behavior.

For weakly acyclic games, decentralized learning algorithms which assign arbitrarily high probabilities to equilibrium policies in the long run are presented in \cite{AY2017}. However, these algorithms do not provide any guarantee on achieving team-optimality when implemented in teams or common interest games. We now strengthen a result of \cite{AY2017} with respect to team-optimality.

\begin{theorem} \label{weaklyTheorem}
Consider a weakly acyclic game in which each DM$^i$ uses Algorithm~\ref{al:main}, and let Assumptions~\ref{noTransientStates}-\ref{as:iup2} hold. For any $\epsilon > 0$, there exist
\begin{gather*}
\tilde{\kappa}_{\epsilon} \in (0,1), \quad
\tilde{\gamma}_{\epsilon}(\bm{\kappa}) \in (0,1), \\
\tilde{W}_{\epsilon}(\bm{\gamma, \kappa}) \in \mathbb{N}_+, \quad
\tilde{T}_{\epsilon}( \bm{\gamma, \kappa} , W_\max) \in \mathbb{N}_+
\end{gather*}
where $W_{\rm max} = \max_i W^i$, such that if, for all $i$, $k\in \mathbb{N}$,
\begin{gather*}
\kappa^i  \in(0,\tilde{\kappa}_{\epsilon}), \quad \gamma^i \in (0, \tilde{\gamma}_{\epsilon}(\bm{\kappa}) ), \\
W^i \geq \tilde{W}_{\epsilon}( \bm{\gamma, \kappa}), \quad T_k \geq \tilde{T}_{\epsilon}( \bm{\gamma, \kappa }, W_\max)
\end{gather*}
then
\begin{equation}
\label{eq:wt}
\liminf_{k\in\mathbb{N}} \P ( \bm{\pi}_k \in \bm{\Pi}_{\rm eq} ) \geq 1 - \epsilon.
\end{equation}
Moreover, if the game is a common interest game, then $\bm{\Pi}_{\rm eq}$ can be replaced by $\bm{\Pi}_{\rm opt}$ in (\ref{eq:wt}).
\end{theorem}

\begin{proof} See Appendix~B. \end{proof}

%%%%%%%%%%%%%%%%%%%%%%%%%%%%%%%%%%%%%%%%%%%%%%%%%%
%%%%%%%%%%%%%%%%%%%%%%%%%%%%%%%%%%%%%%%%%%%%%%%%%%
%%%%%%%%%%%%%%%%%%%%%%%%%%%%%%%%%%%%%%%%%%%%%%%%%%
%%%%%%%%%%%%%%%%%%%%%%%%%%%%%%%%%%%%%%%%%%%%%%%%%%

\section{Learning with Constant Aspirations} \label{sec:constasp}

In this section, we introduce Algorithm~\ref{al:co-main}, a variant of Algorithm~\ref{al:main} in which every DM$^i$ employs a constant aspiration level $\Lambda^i \in \mathbb{R}$ throughout, i.e., $\Lambda^i_k = \Lambda^i$ for every exploration phase $k \in \mathbb{N}$. Pre-setting the aspiration levels is motivated by applications where each DM has the prior knowledge of a conservative estimate of its achievable cost. Such prior knowledge may be available to DMs, for example, from previous experience or through an initial phase of experimentation, and can be used to heuristically discern ``good'' from ``bad'' performance. One implication of this assumption is that if there is indeed a set of joint policies each simultaneously outperforming all pre-set aspiration levels (i.e., the cost estimates) and the other joint policies fail to satisfy any DM, we show that DMs using Algorithm~\ref{al:co-main} will almost surely outperform their aspiration levels in the long run (part (1)-(2) of Theorem~\ref{th:constasp}). This is the case, for example, in a common interest game when the aspiration levels are between the dominant costs and the other costs. In contrast, DMs using Algorithm~\ref{al:main} adaptively adjust their aspiration levels and achieve optimal performance but only in common interest games and in the weaker sense of eventually assigning arbitrarily high probability to the set of optimal policies (Theorem~\ref{theorem1}). In addition, unlike in Algorithm~\ref{al:main}, we characterize the long term behavior of Algorithm~\ref{al:co-main} in all games regardless of whether or not the pre-set aspiration levels are achievable. Loosely speaking, DMs using Algorithm~\ref{al:co-main} in any game are likely to use a certain minimal set of policies in the long run, which are closed under multi-agent strict best-replies (part (3)-(4) of Theorem~\ref{th:constasp}). This minimal set of policies reduces to the set $\Pi_{\rm eq}$ of equilibrium policies in any weakly-acyclic game. Thus, in part (3)-(4) of Theorem~\ref{th:constasp}, we characterize the long-term behavior of Algorithm~\ref{al:co-main}  in a manner analogous to and, in fact, more general than Theorem~\ref{weaklyTheorem}, which characterizes the long-term behavior of Algorithm~\ref{al:main}  as $\bm{\Pi}_{\rm eq}$ in weakly acyclic games.

\begin{algorithm}[h] \label{al:co-main}
\SetAlgoLined
\DontPrintSemicolon
\SetKw{Receive}{Receive}
\SetKw{parameters}{Set Parameters}
\SetKw{initialize}{Initialize}
%\SetKwFor{For}{for}{}{endfor}
\SetKwBlock{For}{for}{end} 

\parameters \;
\Indp 
$\mathbb{J}^i \subset \mathbb{R}^{\mathbb{X}} , \, \mathbb{Q}^i \subset \mathbb{R}^{ \mathbb{X}\times\mathbb{U}^i }$: compact sets \;
$\{ T_k \}_{k \geq 0}$: a sequence in $\mathbb{N}_+$ of exploration phase lengths (common to all DMs) \;
\Indp Set $t_0 = 0$ and $t_{k+1} = t_k + T_k$ for all $k \geq 0.$ \;
\Indm $\rho^i\in(0,1)$: action experimentation probability \;
$\delta^i > 0$: tolerance for sub-optimality when constructing best-reply sets \;
 $\{\gamma_n^i\}_{n\in\mathbb{N}}$, $\kappa^i$: policy experimentation probabilities \;
$\lambda^i\in (0,1)$: inertia parameter for policy update \;
$g^i, h^i \in \mathcal{P} ( \Pi^i | \Pi^i \times 2^{\Pi^i} )$: policy update kernels  \; 
$\Lambda^i \in \mathbb{R}$: an aspiration level \;
$\{\alpha_{n}^{i}\}_{n\geq0}$: step sizes such that $\alpha_{n}^{i}\in[0,1]$, $\sum_n\alpha_n^{i} = \infty$, $\sum_{n} \big(\alpha_n^{i}\big)^2 < \infty$ \;
\Indm 
\BlankLine

\initialize (arbitrary) $\pi_0^i \in \Pi^i$, $J^i_0 \in \mathbb{J}^i$, $Q_0^i\in\mathbb{Q}^i$ \\
\Receive $x_0$ \\

\For($k \geq 0$ ($k^{th}$ exploration phase{)} ){ %i  had to put { and } around ) because otherwise it misunderstands it as ending the for-loops condition. 
	\For( $t = t_k, t_k +1, \dots, t_{k+1} - 1$   ){
	Select  $u^i_t \sim (1-\rho^i) \mathbb{I}_{\pi^i_k ( x_t ) } + \rho^i \uniform ( \mathbb{U}^i)$				\;
	%$u^i_t = 	\begin{cases} 
	%			\pi^i_k ( x_t ), \quad &{\rm w.p. } \quad 1 - \rho^i \\ 
	%			u^i \sim \uniform ( \mathbb{U}^i ), \quad &{\rm w.p. } \quad \rho^i 
	%		\end{cases}$ \;
	\Receive cost $c^i( x_t, u^i_t, \textbf{u}^{-i}_t)$ \;
	\Receive state $x_{t+1} \sim P ( \cdot | x_t, \textbf{u}_t) $ \; 
	Set $m^i_t =$ number of visits to $x_t$ in $[t_k,t]$ \;
	$J^i_{t+1} (x_t) = (1 - \alpha^i_{m^i_t}) J^i_t(x_t)  + \alpha^i_{m^i_t} (c^i(x_t, \textbf{u}_t) + \beta^i J^i_t (x_{t+1}))$\;
	$J^i_{t+1} (x) = J^i_t (x) , \forall x \not= x_t$\;
%	\begin{align*}J^i_{t+1} (x_t) = &(1 - \alpha^i_{m^i_t}) J^i_t(x_t) \\ &\, + \alpha^i_{m^i_t} (c^i(x_t, \textbf{u}_t) + \beta^i J^i_t (x_{t+1}))\end{align*} \;
	Set $n_t^i =$ number of visits to $(x_t,u_t^i)$ in $[t_k, t]$ \;
	$Q_{t+1}^i(x_t,u_t^i)  =   (1-\alpha_{n_t^i}^{i})Q_t^i(x_t,u_t^i) + \alpha_{n_t^i}^{i}  \big[ c^i(x_t,u_t^i, \textbf{u}_t^{-i}) +  \beta^i \min_{v^i} Q_t^i(x_{t+1},v^i) \big]$  \; 
	 $Q_{t+1}^i(x,u^i) =   Q_t^i(x,u^i)$,   $\forall (x,u^i)\not=(x_t,u_t^i)$\; 
	 }
	${\rm BR}^i_k = \{ \pi^i \in \Pi^i : Q^i_{t^i_{k+1}} ( x, \pi^i ( x) ) \leq \min_{v^i} Q^i_{t^i_{k+1}} (x, v^i) + \delta^i, \forall x \in \mathbb{X} \}$ \;
	$\tilde{S}^i_k = \sum_{x \in \mathbb{X}} J^i_{t_{k+1}} ( x) $ \;

	%If suite template from the asynchronous paper: 	
%	\If(\tcp*[f]{Baseline policy update}){ $\pi^i_k \in \BR^i_k$ }{ $\pi^i_{k+1} \leftarrow \pi^i_k$ }
%	\Else(){
%		$\pi^i_{k+1} \leftarrow 	\begin{cases} 	
%								\pi^i_k , &{\rm w.p. } \quad \lambda^i \\
%								\pi^i \sim \uniform ( \BR^i_k ), &{\rm w.p.} \quad 1 - \lambda^i 	
%							\end{cases}$
%		\BlankLine
%	}

	\lIf(){ $\tilde{S}^i_k \leq \Lambda^i$ }{ $\pi^i_{k+1} \sim (1-\gamma^i_k) g^i( \cdot | \pi^i_k , \BR^i_k ) + \gamma^i_k \uniform ( \Pi^i )$}
	\Else(){
	$\pi^i_{k+1} \sim (1 - \kappa^i ) h^i ( \cdot | \pi^i_k , \BR^i_k ) + \kappa^i \uniform ( \Pi^i ) $
	}

	Reset $J^i_{t^i_{k+1}}, Q_{t^i_{k+1}}^i$ to any $J^i \in \mathbb{J}^i, Q^i\in\mathbb{Q}^i$  \;
}

\caption{for DM$^i$}
\end{algorithm}

The following definition is introduced to describe the long-term behavior of Algorithm~\ref{al:co-main}. \vspace*{2mm}
\begin{definition}
For any $i$, $\bm{\eta} \in \bm{\Delta}$, $\bm{\pi} \in \bm{\Pi}$, $\bm{\Lambda} \in \mathbb{R}^N$, let
\[ \tilde{S}^i(\bm{\eta}) := \sum_x J_x^i(\bm{\eta}).\]
\begin{itemize}
%new:
\item[(i)] Let $\widetilde{BR}(\bm{\pi}) :=$
\begin{align*}
 \{\tilde{\bm{\pi}} \in \bm{\Pi} :\tilde{\pi}^i\not=\pi^i \Rightarrow   \tilde{\pi}^i \ \textrm{is a strict best reply to} \ \bm{\pi}, \ \forall  i\}.
\end{align*}
%Old:
%\begin{align*}
%\widetilde{BR}(\bm{\pi}) & := \{\tilde{\bm{\pi}} \in \bm{\Pi} :\tilde{\pi}^i\not=\pi^i \Rightarrow   \tilde{\pi}^i \ \textrm{is a strict best reply} \\ & \qquad\qquad\qquad\qquad\qquad \ \  \textrm{to} \ \bm{\pi}, \ \forall  i\}.
%\end{align*}

A nonempty set of policies $\tilde{\bm{\Pi}}\subset \bm{\Pi}$ is \textit{closed under multi-DM strict best replies}, or a \textit{cumber} set, if
$$\bm{\pi} \in \tilde{\bm{\Pi}} \Rightarrow \widetilde{BR}(\bm{\pi}) \subset \tilde{\bm{\Pi}}.$$
A cumber set is minimal if it does not properly contain another cumber set.
\item[(iii)]
Let
\begin{align*}
%old version, pre Feb 2021: 
%\widetilde{BR}^{\bm{\Lambda}}(\bm{\pi}) & := \{\tilde{\bm{\pi}} \in \bm{\Pi} : \tilde{\pi}^i\not=\pi^i \Rightarrow  \tilde{S}^i(\bm{\pi}) > \Lambda^i  \ \textrm{and} \ \tilde{\pi}^i \ \textrm{is a} \\ & \qquad\qquad \qquad\qquad \ \textrm{strict best reply to} \ \bm{\pi}, \ \forall i\}.
%
%Feb 2021 reformatting: 
\widetilde{BR}^{\bm{\Lambda}}(\bm{\pi}) & := \{\tilde{\bm{\pi}} \in \bm{\Pi} : \tilde{\pi}^i\not=\pi^i \Rightarrow  \tilde{S}^i(\bm{\pi}) > \Lambda^i   \\ &\textrm{and} \ \tilde{\pi}^i \ \textrm{is a} \ \textrm{strict best reply to} \ \bm{\pi}, \ \forall i\}.
\end{align*}

A nonempty set of policies $\tilde{\bm{\Pi}} \subset \bm{\Pi}$ is \textit{closed under multi-DM strict best replies with aspiration levels $\bm{\Lambda} = \{ \Lambda^i \}_{i = 1}^N$}, or a $\bm{\Lambda}$-\textit{cumber} set, if
$$\bm{\pi} \in \tilde{\bm{\Pi}} \Rightarrow \widetilde{BR}^{\bm{\Lambda}}( \bm{\pi} ) \subset \tilde{\bm{\Pi}}.$$
A $\bm{\Lambda}$-cumber set is minimal if it does not properly contain another $\bm{\Lambda}$-cumber set.

\end{itemize}
Let $\bm{\Pi}_{\rm cumber}$ and $\bm{\Pi}_{\rm cumber}^{\bm{\Lambda}} $ denote the union of minimal cumber sets and the union of $\bm{\Lambda}$-minimal cumber sets, respectively.

\end{definition}

The repeated game ($|\mathbb{X}|=1$) with the stage cost functions shown in Figure~\ref{fig2} is a common interest game for $\beta^1=\beta^2$. The minimal cumber sets are $\{$(1,1),(2,1),(2,2),(1,2)$\}$, (which is also a strict best-reply path, and $\{(3,3)\}$, which are also the minimal $\bm{\Lambda}-$cumber sets for $\Lambda^1=\Lambda^2<7$. For $\Lambda^1=\Lambda^2\in[7,10)$, there are three minimal $\Lambda-$cumber sets:
$\{$(2,1)$\}$, $\{$(1,2)$\}$, and $\{(3,3)\}$. For $\Lambda^1=\Lambda^2\in[10,20)$, there are five minimal $\bm{\Lambda}-$cumber sets:
$\{$(1,1)$\}$, $\{$(2,1)$\}$, $\{$(2,2)$\}$, $\{$(1,2)$\}$, and $\{(3,3)\}$. For $\Lambda^1=\Lambda^2\geq20$, any singleton $\{\bm{\pi} \}$, where $\bm{\pi} \in \bm{\Pi}$, is a minimal $\bm{\Lambda}$-cumber set.
On the other hand, for $\Lambda^1\geq10$, $\Lambda^2<7$, the minimal $\bm{\Lambda}-$cumber sets are $\{$(1,1)$\}$, $\{$(2,2)$\}$, and $\{(3,3)\}$.
\begin{figure}[h]
\footnotesize
\centering
\begin{game}{4}{3}[$u_t^1$:][$u_t^2$:]
              & 1 & 2 & 3\\
1     &$10,3$               & $5,7$     & $20,20$\\
2     &$5,7$                & $10,3$    & $20,20$\\
3     &$20,20$              & $20,20$   & $0,0$
\end{game}
\caption[]{Stage cost for a two-DM game where DM$^1$ (DM$^2$) chooses a row (a column) and its cost is the first (the second) entry in the chosen cell.}
\label{fig2}
\end{figure}

Allowing only single-DM best replies in the definition of a cumber set results in the notion of a cusber set introduced in \cite{GEB:stochimitation:2014}.
The following are true, for any $\bm{\Lambda} \in\mathbb{R}^N$.
\begin{itemize}
\item $\bm{\Pi}$ is both a cumber set and a $\bm{\Lambda}$-cumber set.
\item $\bm{\pi}   \in  \bm{\Pi}_{\rm eq}$ $\Leftrightarrow$ $\{ \bm{\pi} \}$ is a (minimal) cumber set.
\item  $\bm{\pi} \in \bm{\Pi}_{\rm eq}$ $\Rightarrow$ $\{ \bm{\pi} \}$ is a (minimal) $\bm{\Lambda}$-cumber set.
\item  ($\bm{\pi} \in \bm{\Pi}$, $\tilde{S}^i( \bm{\pi} ) \leq\Lambda^i$, $\forall i$) $\Rightarrow$ $\{ \bm{\pi} \}$ is a (minimal) $\bm{\Lambda}$-cumber set.
\item There is a multi-DM strict best reply path from any $\bm{\pi} \in \bm{\Pi} \setminus \bm{\Pi}_{\rm cumber}$ to $\bm{\Pi}_{\rm cumber}$.
\item There is a multi-DM strict best reply path from any $\bm{\pi} \in \bm{\Pi} \setminus \bm{\Pi}_{\rm cumber}^{\bm{\Lambda}}$ to $\bm{\Pi}_{\rm cumber}^{\bm{\Lambda}}$.
\item $\bm{\Pi}_{\rm cumber}= \bm{\Pi}_{\rm eq} \Leftrightarrow$ the game is weakly acyclic under multi-DM strict best replies.
\end{itemize}

Let $\bar{L}_{\bm{\pi}}<| \bm{\Pi} |$ be the length of a multi-DM strict best reply path of minimal length from $\bm{\pi} \in \bm{\Pi} \setminus \bm{\Pi}_{\rm cumber}$ to some $\tilde{\bm{\pi}}\in \bm{\Pi}_{\rm cumber}$, and $\bar{L}:=\max_{\bm{\pi} \in \bm{\Pi} \setminus \bm{\Pi}_{\rm cumber}}\bar{L}_{\bm{\pi}}$.

\begin{assumption}
\label{as:constasp}
Assume, for all $i$, $\delta^i \in(0, \bar{\delta})$, $\rho^i \in (0,\rho^{\bm{\Lambda}})$, where $\bar{\delta}$, $\rho^{\bm{\Lambda}}$ are constants defined in Appendix~A,~C, respectively ($\bar{\delta}$ depends only on the game, whereas, $\rho^{\bm{\Lambda}}$ depends on the game and $\bm{\Lambda}$).
Assume further that, for all $i$, $n \in\mathbb{N}$, $\gamma_n^i\in[0,1]$, $\sum_{n \in\mathbb{N}} \gamma^i_n < \infty$, and $\kappa^i \in (0,1)$.
\end{assumption}

\begin{theorem} \label{th:constasp}
Consider a discounted stochastic game where each DM$^i$ updates its policies by Algorithm~\ref{al:co-main}, and let Assumptions~\ref{noTransientStates}~and~\ref{as:constasp} hold.
\begin{enumerate}
\item
Suppose that $g^i = R^{i, 1}$, $\forall i$, and that there exists a nonempty set $\bm{\Pi}^{\bm{\Lambda}} \subset \bm{\Pi}$ satisfying
\begin{equation}
\label{eq:PiL}
\tilde{S}^i(\bm{\pi}^*) < \Lambda^i < \tilde{S}^i(\tilde{ \bm{\pi} }), \ \forall i,  \bm{\pi}^* \in \bm{\Pi}^{\bm{\Lambda}},  \tilde{\bm{\pi}} \in \bm{\Pi} \setminus \bm{\Pi}^{\bm{\Lambda}}.
\end{equation}
Then, there exist $\tilde{T}_k\in\mathbb{N}_+$, $k\in\mathbb{N}$, such that if $T_k\geq\tilde{T}_k$, $\forall k$, then
$$\P \big(\bm{\pi}_k \rightarrow \bm{\pi}^*, \ \textrm{for some} \ \bm{\pi}^* \in \bm{\Pi}^{\bm{\Lambda}}\big)=1.$$

\item
Suppose that $g^i = R^{i, \lambda^i}$, $\forall i$, and that there exists a cumber set $\bm{\Pi}^{\bm{\Lambda}}$ satisfying (\ref{eq:PiL}). Then there exists $\tilde{T}_k\in\mathbb{N}_+$, $k\in\mathbb{N}$, such that if $T_k\geq\tilde{T}_k$, $\forall k$, then
$$\P \big( \bm{\pi}_k \rightarrow \bm{\Pi}^*, \ \textrm{for a minimal cumber set} \ \bm{\Pi}^*\subset \bm{\Pi}^{\bm{\Lambda}} \big)=1.$$

\item
Suppose that $g^i=R^{i,1}$, $h^i=R^{i,\lambda^i}$, $\forall i$. Then,
$$\liminf_{k\in\mathbb{N}} \P ( \bm{\pi}_k \in \bm{\Pi}_{\rm cumber}^{\bm{\Lambda}}) \geq 1- (\bar{L} / \bar{p}_{\min}) \sum_i \kappa^i  $$
for some $\bar{p}_{\min}\in(0,1)$ which is independent of $\sum_i \kappa^i$.

\item
Suppose that $g^i = h^i=R^{i,\lambda^i}$, $\forall i$. Then,
$$\liminf_{k\in\mathbb{N}} \P (\bm{\pi}_k \in \bm{\Pi}_{\rm cumber}) \geq 1-(\bar{L}/\bar{p}_{\min})\sum_i \kappa^i$$
for some $\bar{p}_{\min}\in(0,1)$ which is independent of $\sum_i\kappa^i$.
\end{enumerate}
\end{theorem}

\begin{proof}
See Appendix C.
\end{proof}

Algorithm~\ref{al:co-main} prescribes each DM$^i$ to update its policy differently (using the policy update kernels $g^i$ or $h^i$ coupled with different experimentation probabilities $\gamma_k^i$ or $\kappa^i$) depending on DM$^i$'s assessment of whether its aspiration is achieved or not. The experimentation probability needs to vanish asymptotically for the former case but be positive throughout\footnote{$\kappa^i$ can be time-varying as long as it stay uniformly above zero.} for the latter case.
In practice, the experimentation probabilities for either case are envisioned to be (asymptotically) small so that the policy updates are primarily governed by $g^i$ and $h^i$. With this in mind, Theorem~\ref{th:constasp} can be interpreted as follows.

The first part of Theorem~\ref{th:constasp} assumes (i) each DM$^i$ stays with its policy when it assesses that its aspiration is achieved, (ii) each policy $\bm{\pi} \in \bm{\Pi}$ either simultaneously achieves every DM's aspiration (i.e., $\bm{\pi} \in \bm{\Pi}^{\bm{\Lambda}}$) or not a single DM's aspiration (i.e., $\bm{\pi} \not\in \bm{\Pi}^{\bm{\Lambda}}$). With this (and regardless of $h^i$), DMs converge almost surely to an aspiration achieving joint policy. Note that this does not rule out convergence to a strictly dominated policy.

The second part assumes that the aspiration achieving policies are closed under multi-DM strict best replies. That is, it assumes $\bm{\Pi}^{\bm{\Lambda}}$ is a cumber set. Under this condition (and regardless of $h^i$), DMs converge almost surely to a subset of the aspiration achieving joint policies, which is a minimal cumber set. Note that this rules out neither persistent oscillations within a minimal cumber set (inside the aspiration achieving policies) nor convergence to a set of strictly dominated policies. However, in a weakly-acyclic game (under multi-DM strict best replies), convergence to an aspiration achieving equilibrium is guaranteed; in particular, the equilibrium policies not achieving DMs' aspirations are ruled out. This implies convergence to an optimal policy in teams if the aspiration levels are between the cost of suboptimal and optimal equilibria. If $\bm{\Pi}^{\bm{\Lambda}}$ is not a cumber set, DMs can leave $\bm{\Pi}^{\bm{\Lambda}}$ through multi-DM strict best replies and the result may not hold.

Theorem~\ref{th:constasp} also predicts the long-term behavior of Algorithm~\ref{al:co-main} when the joint policies $\bm{\Pi}$ cannot be partitioned as aspiration achieving policies ($\bm{\Pi}^{\bm{\Lambda}}$) and the other policies in the sense of (\ref{eq:PiL}). The third part of Theorem~\ref{th:constasp} assumes that each DM$^i$ stays with its policy when its aspiration is achieved, otherwise best replies with inertia, i.e., $g^i=R^{i,1}$, $h^i=R^{i,\lambda^i}$. With this (and regardless of the game), DMs' long-term probability of choosing a policy in a minimal $\bm{\Lambda}$-cumber set (a minimal set that DMs cannot exit through the strict best replies of those whose aspirations are not achieved) can be arbitrarily close to one if the experimentation probabilities are sufficiently small. The fourth part assumes that each DM$^i$ always best replies with inertia (when it is not experimenting), i.e., $g^i=h^i=R^{i,\lambda^i}$. With this (and regardless of the game), DMs tend to choose policies in a minimal cumber set (the equilibria and the minimal multi DM strict best reply cycles) for small experimentation probabilities. Under the conditions of the third or the fourth part, DMs may not consistently achieve their aspirations.

\section{A Simulation Study} \label{sec:simulation}

We consider the following two DM stochastic team with $\mathbb{U}^1 = \mathbb{U}^2 = \mathbb{X} = \{1, 2\}$ and common discount factor $\beta = 0.8$. The stage cost for each state is presented in Figure~\ref{fig3}.
\begin{figure}[ht]
\footnotesize
\centering
\begin{game}{3}{2}[$u_t^1$:][$u_t^2$:][$x_t=1$]
& $1$ & $2$ \\
$1$ &$1,1$ &$3,3$ \\
$2$ &$3,3$ &$1,1$
\end{game}
\qquad
\begin{game}{3}{2}[$u_t^1$:][$u_t^2$:][$x_t=2$]
& $1$ & $2$ \\
$1$ &$10,10$ &$10,10$ \\
$2$ &$10,10$ &$13,13$
\end{game}
\caption[]{Stage cost for a two-DM game where DM$^1$ (DM$^2$) chooses a row (a column) and its cost is the first (the second) entry in the chosen cell.}
\label{fig3}
\end{figure}

$x_t=1$ is the low cost state and $x_t=2$ is the high cost state. The transition probabilities, given below, are constructed so that when DMs successfully coordinate their decisions (in a state-dependent manner) the state transitions with high probability to the low cost state. Otherwise, the state transitions with high probability to the high cost state. The transition kernel is fully described by 
\begin{align*}
&P ( 1 \vert x, a^1, a^2 ) = 0.95, \text{ if } x = a^1 = a^2 \\
&P ( 2 \vert x, a^1, a^2 )  = 0.95, \text{ if } x \not= a^1 \text{ or } a^1 \not= a^2
\end{align*}
%Below: old version, with abused notation. I'm changing it.
%\begin{align*}
%&P ( x_{t+1} = 1 \vert x_t = u^1_t = u^2_t ) = 0.95 \\
%&P ( x_{t+1} = 2 \vert x_t \not = u^1_t \ \textrm{or} \ u^1_t \not= u^2_t) = 0.95
%\end{align*}
In particular, when $x_t=2$, DMs are faced with the choice between on the one hand incurring a lower short term cost $10$ and likely remaining in the high cost state and on the other hand paying a higher short term cost $13$ with the hopes of transitioning to the low cost state and avoiding sustained high costs.

For sufficiently large discount factors (including $\beta = 0.8$ as selected), the unique team optimal policy is for both DMs to coordinate as $u_t^1=u_t^2=x_t$, for all $t\in\mathbb{N}$. %However, there are three suboptimal equilibrium policies (namely, that where both players always play action 1, that where both players always play action 2, and that where players always play the action whose label does not match the state--i.e. playing action 2 in state 1 and playing action 1 in state 2.).
However, there are three suboptimal equilibrium policies, namely (i) $u_t^1=u_t^2=1$, for all $t\in\mathbb{N}$, (ii) $u_t^1=u_t^2=2$, for all $t\in\mathbb{N}$, (iii) $u_t^1=u_t^2\not=x_t$, for all $t\in\mathbb{N}$.

We simulated Algorithm~\ref{al:main}~and~\ref{al:co-main} with the following parameter choices.
\begin{align*}
\textrm{Case A:} \  \textrm{Algorithm~\ref{al:main},} & \ h^i=R^{i,\lambda^i}, \ \lambda^i\in(0,1) \\ & \ \kappa^i=\gamma^i+0.1, \ W^i=30, \ T_k=10000 \\
\textrm{Case B:} \  \textrm{Algorithm~\ref{al:main},} & \ h^i=R^{i,\lambda^i}, \ \lambda^i=1 \\ & \ \kappa^i=1, \ W^i=50, \ T_k=5000 \\
\textrm{Case C:} \  \textrm{Algorithm~\ref{al:co-main},} & \ g^i=h^i=R^{i,\lambda^i}, \ \lambda^i\in(0,1), \Lambda^i = 30 \\ & \ \kappa^i=\gamma^i+0.2, \ T_k=7500 \\
\textrm{Case D:} \  \textrm{Algorithm~\ref{al:co-main},} & \ g^i=h^i=R^{i,\lambda^i}, \ \lambda^i=1 , \Lambda^i = 30 \\ & \ \kappa^i=\gamma^i+0.2, \ T_k=7500
\end{align*}
where the aspiration level $\Lambda^i$ used in case C and D was chosen without extensive tuning.

The algorithms performed generally as expected. The disparity across different cases owes largely to the parameter selections.
In each case, the percentage of time where the joint policies are team optimal, i.e., $\bm{\pi}_k\in \bm{\Pi}_{\rm opt}$, are shown below.

\begin{center}
\begin{tabular}{ | c | c | c | c | c |}
\hline
Case 	& $\gamma = 0.05$ 	& $\gamma = 0.01$ 	& $ \gamma = 0.005$ 	& $ \gamma = 0.001$ \\
\hline
A 		&0.638  				& 0.902				& 0.922				& 0.972		\\
\hline
B		& 0.432  				& 0.776				& 0.864				& 0.952	\\
\hline
C 		&  0.648				& 0.908				& 0.960				& 0.984	\\
\hline
D		&  0.242				& 0.564				& 0.720				& 0.914 		\\
\hline
\end{tabular}
\end{center}

As the experimentation probability $\gamma$ is reduced, the empirical frequency of the event $\bm{\pi}_k\in \bm{\Pi}_{\rm opt}$ increases and, for $\gamma=0.001$, the joint policies are team optimal for more than $90\%$ of the time. These numerical results confirm the theoretical results.

\section{Concluding Remarks} \label{sec:conclusion}

In this paper, we presented learning algorithms for stochastic teams and common interest games under a decentralized information structure in which players do not share actions with one another. While previous studies have focused on repeated games, or otherwise used a large degree of control sharing among decision makers to obtain convergence results, we have provided a method for achieving team optimality in teams and stochastic common interest without any control sharing during play and with limited prior information about the game.

The proof methods used in this paper center on approximating the true joint policy-valued stochastic process using time homogenous Markov chains through a novel Dobrushin's coefficient based analysis. The algorithms presented are amenable to further variations and can be modified as needed, and the Markov chain analysis used for the convergence guarantees can likewise be easily modified for more general applications.

We chose to focus on games with full state observations available to each agent since there are few formal results on multi-agent learning even under this simplifying assumption. The partially observed information structure, in which each player has access to only local state information, is an important and challenging direction for future research.

%%%%%%%%%%%%%%%%%%%%%%%%%%%%%%%%%%%%%%%%%%%%%%%%%%%%%%%%%%%%%%%%%%%%%%%%%%%%%%%%%%%%%%%%%%%%%%%%%%%%%%%%%%%%%%%%%%%%%%%%%%%%%%%%%%%%%%%%%%%%%%%%%%%%%%%%%%%%%%%%%%%%%%%%%%%%%%%%%%%%%%%%%%%%%%%%%%%%%%%%%%%%%%%%%%%%%%%%%%%%%%%%%%%%%%%%%%%%%%%%%%%%%%%%%%%%%%%%%%%%%%%%%%%%%%%%%%%%%%%%%%%%%%%%%%%%%%%%%%%%%%%%%%%%%%%%%%%%%%%%%%%%%%%%%%%%%%%%%%%%%%%%%%%%%%%%%%%%%%%%%%%%%%%%%%%%%%%%%%%%%%%%%%%%%%%%%%%%%%%%%%%%%%%%%%%%%%%%%%%%%%%%%%%%%%%%%%%%%%%%%%%%%%%%%%%%%%%%%%%%%%%%%%%%%%%%%%%%%%%%%%%%%%%%%%%%%%%%%%%%%%%%%

\section*{Appendix A: Proof of Theorem~\ref{theorem1}}
Let $\sigma(A)\in[0,1]$ denote the Dobrushin coefficient  of an $n \times n$ right stochastic matrix $A$, defined as \cite{dobrushin}
\begin{equation}
\label{eq:dobrushin}
\sigma(A) := \min_{i, k \in \{ 1, \dots, n \}}  \sum_{j = 1}^n \min \{ A(i,j) , A(k, j) \}.
\end{equation}

\begin{lemma} \label{irreducibleMatrixApproxLemma}
Consider an $n\times n$ right stochastic matrix $A$ with $\sigma(A) > 0$, and a sequence of $n\times n$ right stochastic matrices $\{ A_k \}_{k \in \mathbb{N}}$.
For any $\epsilon \in(0,1)$, if
\begin{equation}
\label{eq:Ak-A}
\sup_{k\in\mathbb{N}} \| A_k - A \|_{\infty} \leq \tau := \frac{\sigma(A) \epsilon}{2n}
\end{equation}
then, for any probability vector $\mu_0$ of dimension $n$,
\[
\limsup_{k\in\mathbb{N}} \| \mu_0 A_0 \cdots A_k -  \mu^* \|_1 \leq \epsilon
\]
where $\mu^*$ is the unique probability vector satisfying $\mu^*=\mu^* A$.
\end{lemma}

\begin{proof}
Recall that $\| \mu A - \nu A \|_1 \leq (1 - \sigma(A) ) \| \mu - \nu \|_1$, for all probability vectors $\mu$, $\nu$; see \cite{dobrushin}.
Since $\sigma(A) > 0$, by Banach's fixed point theorem, there exists a unique probability vector $\mu^*$ satisfying $\mu^* = \mu^* A$, and $\lim_k\mu_0 A^k = \mu^*$, for any probability vector $\mu_0$.

From (\ref{eq:dobrushin})-(\ref{eq:Ak-A}), we have $ \sup_{k\in\mathbb{N}} \vert \sigma(A_k) - \sigma(A) \vert \leq n \tau$, which implies $\sup_{k\in\mathbb{N}} (1 - \sigma(A_k))  \leq \xi := 1 - \sigma(A)/2 .$ Note $\xi \in (0, 1)$.
We write
\begin{align*}
\| \mu_0 A_0 - \mu^* \|_1	&= \| \mu_0 A_0 - \mu^* A \|_1  \\
					&\leq \| \mu_0 A_0 - \mu^* A_0 \|_1 + \| \mu^* A_0 - \mu^* A \|_1 \\
					&\leq (1 - \sigma(A_0) ) \| \mu_0 - \mu^* \|_1 + n \tau \\
					&\leq \xi \| \mu_0 - \mu^* \|_1 + n \tau
\end{align*}
Repeated application of these inequalities result in
\begin{align*}
\| \mu_0 A_0\dots A_{k-1} - \mu^* \|_1 \leq & \xi^k \| \mu_0 - \mu^* \|_1 + \epsilon, \quad \forall k,
\end{align*}
where $\epsilon = n \tau \frac{1}{1-\xi}$, which is consistent with (\ref{eq:Ak-A}). As $\lim_k \xi^k \| \mu_0 - \mu^* \|_1 = 0$, the lemma follows.
\end{proof}

\subsection*{Proof of Theorem \ref{theorem1}}
%\bmy{I am replacing $\xi$ in the proof of theorem 1 by $\tau$, so I can use $\xi$ in the proof of the weakly acyclic generalization.}
Let $\epsilon \in (0,1)$ and $\bm{\kappa} \in (0, 1)^N$. By Lemma \ref{IdealizedProcessLemma}, there exists $\bar{\gamma}_{\epsilon}(\bm{\kappa})$ such that $\max_i \gamma^i \in ( 0, \bar{\gamma}_{\epsilon}(\bm{\kappa}))$ implies $\mu_{\bm{\gamma,\kappa, h}}^* ( \bm{\Pi}_{\rm opt}) \geq 1 - \epsilon/2$, where $\mu_{\bm{\gamma, \kappa, h}}^*$ is the unique invariant measure of the Markov chain induced by the IUP. Assume $\max_i \gamma^i\in (0, \bar{\gamma}_{\epsilon}(\bm{\kappa}))$.

For all $k \in \mathbb{N}$, $\bm{\pi}$, $\bm{\pi}^{\prime} \in \bm{\Pi}$, we define
\begin{align}
\mu_k ( \bm{\pi} )     & := \P( \bm{\pi}_k = \bm{\pi} ) \label{eq:alg1mu}\\
A_k ( \bm{\pi}, \bm{\pi}^{\prime}) & := \P( \bm{\pi}_{k+1} = \bm{\pi}^{\prime} \vert \bm{\pi}_k = \bm{\pi})  \label{eq:alg1Ak}
\end{align}
where $\bm{\pi}_k$ is the joint baseline policy during the $k^{th}$ exploration phase of Algorithm~\ref{al:main}. Note that $\mu_{k+1} = \mu_0 A_0 \cdots A_k$. To prove the theorem, we will show $$\limsup_{k\in\mathbb{N}} \|\mu_k-\mu_{\bm{\gamma, \kappa, h}}^*\|_1 \leq \epsilon/2.$$

\noindent Due to Lemma~\ref{irreducibleMatrixApproxLemma} and $\sigma(A_{\bm{\gamma,\kappa, h}})>0$~\footnote{
		$A_{\bm{\gamma, \kappa, h}} ( \bm{\pi}, \bm{\pi}' ) \geq \prod_{i = 1}^N \min \{ \gamma^i, \kappa^i \} / | \Pi^i | > 0$, $\forall \bm{\pi}, \bm{\pi}' \in \bm{\Pi}$, due to uniform experimentation by each DM$^i$ with probability $\gamma^i>0$ or $\kappa^i>0$. By (\ref{eq:dobrushin}), this implies $\sigma ( A_{\bm{\gamma, \kappa, h}} )>0$.},
it is sufficient to show
\begin{equation} \label{CkCloseToA}
\| A_k - A_{\bm{\gamma,\kappa, h}} \|_{\infty} \leq \tau := \frac{ \sigma ( A_{\bm{\gamma,\kappa , h} }) \epsilon }{4 \vert \bm{\Pi} \vert}
\end{equation}
for all but finitely many $k\in\mathbb{N}$. We note that $\sigma ( A_{\bm{\gamma, \kappa, h} } ) > 0$ since all entries of $A_{\bm{\gamma, \kappa, h} }$ are strictly positive, as the IUP updates policies using uniform randomization with strictly positive probability owing to $\gamma^i, \kappa^i > 0$ for every DM$^i$.

To ensure (\ref{CkCloseToA}), we will introduce an event $R_k$ such that, for all $\bm{\pi}$, $\bm{\pi}^{\prime} \in \bm{\Pi}$, and all but finitely many $k\in\mathbb{N}$,
\begin{equation} \label{pR_k}
\P ( \bm{\pi}_{k+1} = \bm{\pi}^{\prime} \vert \bm{\pi}_k = \bm{\pi} , R_k ) = A_{\bm{\gamma,\kappa, h}} ( \bm{\pi}, \bm{\pi}^{\prime})
\end{equation}
and we will show that
\begin{equation} \label{conditionalR_k}
\P ( R_k \vert \bm{\pi}_k = \bm{\pi} ) \geq 1 - \tau
\end{equation}
by choosing the parameters of Algorithm~\ref{al:main} appropriately. Note that (\ref{pR_k})-(\ref{conditionalR_k}) imply (\ref{CkCloseToA}) as follows:
\begin{align*}
&A_{ \bm{\gamma,\kappa, h}} ( \bm{\pi}, \bm{\pi}^{\prime} ) - \tau \\
&\leq A_{ \bm{\gamma,\kappa, h}} ( \bm{\pi}, \bm{\pi}^{\prime} ) ( 1 - \tau ) \\
&\leq A_{\bm{\gamma,\kappa, h}} ( \bm{\pi}, \bm{\pi}^{\prime} ) \P ( R_k \vert \bm{\pi}_k = \bm{\pi} ) \\ % 2nd ineq: by (12) + positive term
& \quad + Pr ( \bm{\pi}_{k+1} = \bm{\pi}^{\prime} \vert \bm{\pi}_k = \bm{\pi}, R_k^c ) \P ( R_k^c \vert \bm{\pi}_k = \bm{\pi})  \\ %this quality holds by (11)
&= A_k ( \bm{\pi},  \bm{\pi}^{\prime} ) \\ %by the law of total probability
&\leq A_{ \bm{\gamma,\kappa, h}} ( \bm{\pi}, \bm{\pi}^{\prime} ) \cdot 1 + 1 \cdot P ( R_k^c \vert \bm{\pi}_k = \bm{\pi} ) \\
&\leq A_{ \bm{\gamma,\kappa, h}} ( \bm{\pi}, \bm{\pi}^{\prime} ) + \tau
\end{align*}
where $R_k^c$ denotes the complement of $R_k$.

Define
\begin{align}
\bar{\delta} &:= \min  \{ \vert Q^{*i}_{ \bm{\pi}^{-i} } ( x, u ) - Q^{*i}_{ \bm{\pi}^{-i} } ( x, v) \vert >0  :  \nonumber \\
& \qquad\qquad \ i, \bm{\pi}^{-i} \in \bm{\Pi}^{-i} , x \in \mathbb{X} , u, v \in \mathbb{U}^i \} \label{eq:deltabar}\\
S^i( \bm{\pi} ) & := \sum_{x \in \mathbb{X}} Q^{*i}_{ \bm{\pi}^{-i}} ( x, \pi^i(x)), \quad \forall  \bm{\pi} \in \bm{\Delta} \nonumber \\
\bar{d} &:= \frac{1}{2}\min \{ |S^i( \bm{\pi} )-S^i(\tilde{ \bm{\pi} })|>0 : i, \bm{\pi} , \tilde{\bm{\pi}}\in\bm{\Pi} \}. \nonumber
\end{align}

\noindent Let $\bar{\pi}_k^i\in\Delta^i$ denote the policy used by DM$^i$ in the $k^{th}$ exploration phase, i.e., $$\bar{\pi}_k^i(\cdot | x) :=(1-\rho^i)\mathbb{I}_{\pi_k^i(x)}+\rho^i \uniform(\mathbb{U}^i) , \quad \forall x\in\mathbb{X} .$$
\noindent Let $\bar{\rho}>0$ be such that $\max_i \rho^i\in(0,\bar{\rho})$ implies
\begin{align*}
\| Q^{*i}_{ \bm{\pi}^{-i}_k } - Q^{*i}_{\bar{\bm{\pi}}^{-i}_k } \|_\infty  & < \frac{1}{2} \min \{ \delta^i , \bar{\delta} - \delta^i\}, \quad \forall i, k\in\mathbb{N} \\
| S^i( \bm{\pi}_k) - S^i(\bar{\bm{\pi}}_k)  |  & < \frac{1}{2} \min \{ d^i, \bar{d} - d^i \}, \quad \forall i, k\in\mathbb{N}.
\end{align*}
Such $\bar{\rho}$ exists due to \cite[Lemma 3]{AY2017}.
Assume that, for all $i$, $d^i \in(0, \bar{d})$, $\delta^i \in (0, \bar{\delta})$, $\rho^i \in (0, \bar{\rho})$. Assume further that
$$W^i \geq \bar{W}_{\epsilon} (\bm{\gamma,\kappa}) := \min \{ W \in \mathbb{N} : (1-\phi)^W < \phi \tau/3 \}, \ \forall i$$
where  $\phi:=\prod_i \min\{\gamma^i/|\Pi^i|,\kappa^i/|\Pi^i|\}\in(0,1)$.

For any time $k \geq W_{\max}$, we define the event $$R_k  := F_k \cap \bigcap_{\ell =0 }^{W_\max} G_{k-\ell} \cap \bigcup_{\ell = 1}^{\bar{W}_{\epsilon}( \bm{\gamma,\kappa})} H_{k-\ell}$$ where, for any $\ell \in \mathbb{N}$, we define
\begin{align*}
F_\ell &:= \{ \| Q^i_{t_{\ell+1}} - Q^{*i}_{\bm{\pi}^{-i}_\ell } \|_\infty <\min_i \{ \delta^i , \bar{\delta} - \delta^i \} /2, \forall i \} \\
G_\ell &:=  \{ \vert S^i_\ell - S^i(\bm{\pi}_\ell) ) \vert < \min\{d^i,\bar{d}-d^i\}/2, \forall i \} \\
H_\ell &:=  \{ \bm{\pi}_\ell \in \bm{\Pi}_{\rm opt}  \}.
\end{align*}
Conditioned on $R_k$, $k \geq W_\max$, we have, for all $i$,
$$\BR^i_k = \BR^i ( \bm{\pi}^{-i}_k)$$
and
$$S^i_k \leq \Lambda^i_k  \iff \bm{\pi}_k \in \bm{\Pi}_{\rm opt}.$$
This implies (\ref{pR_k}), for all $k \geq W_\max$. Intuitively, the event $R_k$ guarantees that (i) all players have sufficiently reliable Q-factors during the $k^{th}$ exploration phase, due to $F_k$; (ii) for every DM$^i$, the estimated cost scores are sufficiently close to the true cost scores during every exploration phase in DM$^i$'s most recent memory window, by $G_k$; (iii) an optimal baseline policy was visited recently enough that all players remember its cost score, by $H_k$.

We will now show (\ref{conditionalR_k}) for sufficiently large exploration lengths $\{ T_\ell \}_\ell$. (Since the events $G_\ell$ and $F_\ell$ are defined in terms of Q-factors, this is a statement about the long term behavior of the Q-factor iterates \textit{within an exploration phase}.)
 
Note that within the $k^{th}$ exploration phase of Algorithm~\ref{al:main} (and Algorithm~\ref{al:co-main}), the environment faced by each DM$^i$, that is determined by $\bm{\pi}^{-i}_k$, is a stationary MDP (with finite state and control spaces) and satisfies the usual conditions of stochastic approximation theory.
In such a setting, it is well-known that the sequence of Q-factors produced by the standard Q-learning algorithm from any initial condition is bounded and convergent with probability one \cite{tsitsiklis}.
Since each exploration phase in Algorithm~\ref{al:main} (and in Algorithm~\ref{al:co-main}) starts with re-initialized Q-factors (and what we may call J-factors in the case of Algorithm~\ref{al:co-main}) within the compact sets $\{\mathbb{Q}^i\}_{i=1}^N$ (and $\{\mathbb{J}^i\}_{i=1}^N$), the Q-factors (and J-factors) produced by Algorithm~\ref{al:main} (and Algorithm~\ref{al:co-main}) during any exploration phase remain bounded with probability one; c.f. Lemma 1 in \cite{AY2017}.

Furthermore, Lemma~1 in \cite{AY2017} shows that, uniformly in the initial conditions within $\{\mathbb{Q}^i\}_{i=1}^N$ (and $\{\mathbb{J}^i\}_{i=1}^N$), the Q-factors (and J-factors) produced by Algorithm~\ref{al:main} (and Algorithm~\ref{al:co-main}) enter any arbitrarily small neighborhood of their limits with arbitrarily high probability at the end of any sufficiently long exploration phase.

By \cite[Lemma~4]{AY2017}, there exists $T_{\epsilon}( \bm{\gamma,\kappa},W_\max) \in \mathbb{N}_+$ %\bmy{($L$ a function of $\epsilon, (f^i)_{i=1}^N, (d^i)_{i=1}^N, (\delta^i)_{i=1}^N, W$)}
such that if $\min_{k\in\mathbb{N}} T_k \geq T_{\epsilon}(\bm{\gamma,\kappa},W_\max)$, we have
$$\P(F_k), \P(G_k) \geq 1- \phi \tau /(3 W_\max).$$
This implies, for $k \geq W_\max$,
\begin{align*}
\P ( \cap_{\ell = 0}^{W_\max} G_{k - \ell}) \geq  1-\phi\tau/3.
 \end{align*}
In addition, we have, for $k \geq W_\max$,
\begin{align*}
\P ( H_k ) \geq & 1 - (1 - \phi)^{\bar{W}_{\epsilon}(\bm{\gamma,\kappa})} \geq 1 - \phi \tau/3.
\end{align*}
All together, the preceding imply, for $k \geq W_\max$,
\begin{equation} \label{RkEquation}
\P (R_k ) \geq 1 - \phi \tau .
\end{equation}
Since $\inf_{k\in\mathbb{N}, \bm{\pi} \in \bm{\Pi}} \P ( \bm{\pi}_k =  \bm{\pi} ) \geq \phi >0 $, (\ref{RkEquation}) implies (\ref{conditionalR_k}), because  (\ref{RkEquation}) implies $\P ( R_k \cap \{ \bm{\pi}_k = \bm{\pi} \} ) \geq (1 - \tau ) \P ( \bm{\pi}_k = \bm{\pi} )$ for any $k, \bm{\pi}$.
%For us: this is because (14) implies Pr ( R_k \cap { pi_k = pi } ) \geq (1 - \tau ) Pr ( pi_k = pi ).

We have shown that (\ref{pR_k}) and (\ref{conditionalR_k}) hold. In turn, this implies (\ref{CkCloseToA}) holds, and invoking Lemma \ref{irreducibleMatrixApproxLemma} completes the proof. \hfill $\square$

\section*{Appendix B: Proof of Theorem \ref{weaklyTheorem}}

\begin{lemma} \label{irreducibleMatrixApproxLemma2}
Consider an $n\times n$ right stochastic matrix $A$, and a sequence of $n\times n$ right stochastic matrices $\{ A_k \}_{k \in \mathbb{N}}$.
For any $\epsilon \in(0,1)$, $m\in\mathbb{N}$, if
\begin{equation}
\label{eq:Ak-A2}
\sup_{k\in\mathbb{N}} \| A_k - A \|_{\infty} \leq \epsilon/(2nm)
\end{equation}
then
$$\sup_{k\in\mathbb{N},\mu_0} \| \mu_0 A_k \cdots A_{k+m-1} -  \mu_0 A^m \|_1 \leq \epsilon/2$$
where $\mu_0$ is any probability vector of dimension $n$.
\end{lemma}

\subsection*{Proof of Theorem \ref{weaklyTheorem}}

Let $\epsilon \in (0, 1)$. Assume
$$0<\kappa^i<\tilde{\kappa}_{\epsilon}:=\min\{\bar{\kappa}_{\epsilon},\epsilon/(4| \bm{\Pi} |\bar{m}N)\}, \quad \forall i$$
where $\bar{\kappa}_{\epsilon}$ and $\bar{m}$ are as in Lemma~\ref{IdealizedProcessLemma2}. Then, assume
$$0<\gamma^i<\tilde{\gamma}_{\epsilon}(\bm{\kappa}):=\min\{\bar{\gamma}_{\epsilon}( \bm{\kappa}),\tilde{\kappa}_{\epsilon}\}, \quad \forall i$$
where $\bar{\gamma}_{\epsilon}( \bm{\kappa})$ is as in Lemma~\ref{IdealizedProcessLemma}.
With these choices of $\bm{\gamma}$, $\bm{\kappa}$, Lemma~\ref{IdealizedProcessLemma2} holds, i.e.,
\begin{equation}
\inf_{\mu_0\in\mathcal{P}(\bm{\Pi})} (\mu_0 A_{\bm{\gamma, \kappa}}^{\bar{m}})(\bm{\Pi}_{\rm eq}) \geq 1-\epsilon/2. \label{eq:watp3}
\end{equation}

For any $k \in \mathbb{N}$, defining $A_k$ as in \eqref{eq:alg1Ak}, we have
\begin{align}
& \|A_k - A_{\bm{\gamma, \kappa,h}} \|_{\infty} \leq 1 - \prod_i  (1 - \max\{\gamma^i,\kappa^i\}) \nonumber \\
& \qquad \qquad \qquad \qquad \times \min_{\pi\in\Pi} \P \big( \BR_k^i= \BR^i( \bm{\pi}^{-i}), \forall i | \bm{\pi}_k= \bm{\pi} \big) \label{eq:watp1}
\end{align}
%where $A_k$ is as in (\ref{eq:alg1Ak}).
By \cite[Lemma~4]{AY2017}, there exists $T_\epsilon\in\mathbb{N}_+$ such that if $T_k\geq T_\epsilon$,
\begin{equation}
\min_{\bm{\pi} \in \bm{\Pi}} \P \big( \BR_k^i = \BR^i(\bm{\pi}^{-i}), \forall i | \bm{\pi}_k= \bm{\pi} \big) \geq 1 - \frac{\epsilon}{4| \bm{\Pi} |\bar{m}N}. \label{eq:watp2}
\end{equation}
Assume that, for all $i$, $k\in\mathbb{N}$,
\begin{align*}
W^i & \geq\tilde{W}_{\epsilon}( \bm{\gamma,\kappa} ) :=  \bar{W}_{\epsilon}( \bm{\gamma,\kappa} ) \\
T_k & \geq\tilde{T}_\epsilon( \bm{\gamma,\kappa}, W_\max) :=  \max\{T_\epsilon , \bar{T}_\epsilon( \bm{\gamma,\kappa},W_\max) \}
\end{align*}
where $\bar{W}_{\epsilon}( \bm{\gamma,\kappa})$, $\bar{T}_\epsilon(\bm{\gamma,\kappa},W_\max)$ are as in Theorem~\ref{theorem1}.
By (\ref{eq:watp1})-(\ref{eq:watp2}) and the assumptions on $\bm{\gamma}$, $\bm{\kappa}$, $\{T_k\}_{k\in\mathbb{N}}$ , we have
$$\sup_{k\in\mathbb{N}} \|A_k - A_{\bm{\gamma, \kappa}} \|_{\infty} \leq \epsilon/(2 | \bm{\Pi} | \bar{m}N).$$
Lemma~\ref{irreducibleMatrixApproxLemma2} implies
\begin{equation}
\label{eq:pfwad}
\sup_{k\in\mathbb{N},\mu_0\in\mathcal{P}(\Pi)} \| \mu_0 A_k \cdots A_{k+\bar{m}-1} -  \mu_0 A_{\bm{\gamma,\kappa}}^{\bar{m}} \|_1 \leq \epsilon/2.
\end{equation}
The desired result for weakly acyclic games follows from (\ref{eq:watp3})-(\ref{eq:pfwad}). Note that the parameter choices satisfy the hypothesis of Theorem~\ref{theorem1}; hence,
the results of Theorem~\ref{theorem1} also hold.

\section*{Appendix C: Proof of Theorem~\ref{th:constasp}}
\label{appx:constasp}
Let $\bar{\delta}^i$ be as in (\ref{eq:deltabar}), and let $\rho^{\bm{\Lambda}}\in(0,1)$ be such that $\rho^i\in(0,\rho^{\bm{\Lambda}})$, for all $i$, implies
$$ \| Q^{*i}_{ \bm{\pi}^{-i}_k } - Q^{*i}_{\bar{ \bm{\pi} }^{-i}_k } \|_\infty < \frac{1}{2} \min \{ \delta^i , \bar{\delta} - \delta^i \}, \quad \forall i, k\in\mathbb{N}$$
and
$$ | \tilde{S}^i( \bm{\pi}_k) - \tilde{S}^i(\bar{\bm{\pi}}_k)|  <  \min_{ \bm{\pi} \in \bm{\Pi}} | \Lambda^i - S^i(\bm{\pi})|, \quad \forall i, k\in\mathbb{N}$$
where $\bar{\pi}^i_k(\cdot | x_k)=(1 - \rho^i) \mathbb{I}_{\pi^i_k(x_k)} + \rho^i \uniform (\mathbb{U}^i)$. Such $\rho^{\bm{\Lambda}} \in(0,1)$ exists due to \cite[Lemma 3]{AY2017}.

Let $\epsilon_k\in(0,1)$, $k\in\mathbb{N}$, be such that  $\sum_{k\in\mathbb{N}} \epsilon_k < \infty$.
Due to \cite[Lemma 1]{AY2017}, there exists finite integers $\tilde{T}_k\in\mathbb{N}_+$, $k\in\mathbb{N}$ such that if $T_k\geq\tilde{T}_k$, for all $k\in\mathbb{N}$,
$$\P(|\tilde{S}_k^i-\tilde{S}^i(\bar{ \bm{\pi} }_k)|<\epsilon_k, \|Q_{t_{k+1}}^i-Q_{\bar{ \bm{\pi} }_k^{-i}}^{*i} \|_{\infty}<\epsilon_k, \forall i) \geq 1 - \epsilon_k.$$
Assume now $T_k\geq\tilde{T}_k$, for all $k\in\mathbb{N}$.
Hence, there exists $\tilde{k}\in\mathbb{N}_+$ such that, for all $i$, $k\geq\tilde{k}$,
$$\P (E_k) \geq  1 - \epsilon_k$$
where
\begin{align*}
E_k := & \{(( \bm{\pi}_k\in \bm{\Pi}^{\bm{\Lambda}}, \tilde{S}_k^i < \Lambda^i) \ \textrm{or} \ (\bm{\pi}_k \not\in \bm{\Pi}^{\bm{\Lambda}}, \tilde{S}_k^i > \Lambda^i)), \\ &  \qquad \BR_k^i = \BR^i( \bm{\pi}_k^{-i}), \ \forall i \}.
\end{align*}
\begin{enumerate}
\item
We have, for all $k\geq\tilde{k}$,
\begin{align*}
\P( \bm{\pi}_{k+1}\in \bm{\Pi}^{\bm{\Lambda}} | \bm{\pi}_{k} \in \bm{\Pi}^{\bm{\Lambda}}) & \geq (1-\epsilon_k)\prod_i (1-\gamma_k^i) \\
\P( \bm{\pi}_{k+1}\in \bm{\Pi}^{\bm{\Lambda}} | \bm{\pi}_{k} \notin \bm{\Pi}^{\bm{\Lambda}})& \geq (1-\epsilon_k) \prod_i  (\kappa^i/|\Pi^i|).
\end{align*}
This leads to, with some algebra, for all $k\geq\tilde{k}$,
\begin{align*}
\P(\bm{\pi}_{k+1}\not\in\bm{\Pi}^{\bm{\Lambda}}) \leq & \Big( 1 -  \prod_i (\kappa^i/|\Pi^i|) \Big) \P( \bm{\pi}_k\not\in \bm{\Pi}^{\bm{\Lambda}}) \\
& + \epsilon_k + \sum_i \gamma_k^i.
\end{align*}
Since $\Big|1 -  \prod_i (\kappa^i/|\Pi^i|)\Big|<1$, $\sum_{k\in\mathbb{N}} \epsilon_k < \infty$, and $\sum_{i,k\in\mathbb{N}} \gamma_k^i<\infty$, we have $\sum_{k\in\mathbb{N}} \P( \bm{\pi}_k\not\in \bm{\Pi}^{\bm{\Lambda}}) < \infty$. Borel-Cantelli Lemma implies
$$\P( \bm{\pi}_k\not\in \bm{\Pi}^{\bm{\Lambda}}, \ \textrm{for infinitely many} \ k\in\mathbb{N}) = 0.$$
Also, $\sum_{k\in\mathbb{N}} \P(( \bm{\pi}_k\in \bm{\Pi}^{\bm{\Lambda}}, \tilde{S}_k^i \geq \Lambda^i)<\infty$, hence
$$\P(\bm{\pi}_k \in \bm{\Pi}^{\bm{\Lambda}}, S_k^i \geq \Lambda^i, \ \textrm{for infinitely many} \ k\in\mathbb{N})  = 0.$$
This proves the first part. 

\item
We have, for all $k\geq\tilde{k}$,
\begin{align*}
& \P(\bm{\pi}_{k+1}\in \bm{\Pi}^{\bm{\Lambda}} \cap \bm{\Pi}_{\rm cumber} | \bm{\pi}_k \in \bm{\Pi}^{\bm{\Lambda}}\cap \bm{\Pi}_{\rm cumber}) \\
& \qquad \geq  (1-\epsilon_k) \prod_i (1-\gamma_k^i).
\end{align*}
There exists $\bar{p}_{\min}\in(0,1)$ (which depends only on $\lambda^1,\dots,\lambda^N$, $|\Pi^1|,\dots,|\Pi^N|$, $\bar{L}$) such that, for all $k\geq\tilde{k}$,
\begin{align}
\nonumber
& \P( \bm{\pi}_{k+\bar{L}} \in \bm{\Pi}^{\bm{\Lambda}} \cap \bm{\Pi}_{\rm cumber} | \bm{\pi}_k \in \bm{\Pi}^{\bm{\Lambda}} \setminus \bm{\Pi}_{\rm cumber}) \\ & \qquad \geq \bar{p}_{\min} \prod_{n=k}^{k+\bar{L}-1} (1-\epsilon_n) \prod_i (1-\gamma_n^i) \label{eq:pmin}
\end{align}
and
\begin{align*}
& \P( \bm{\pi}_{k+\bar{L}}\in \bm{\Pi}^{\bm{\Lambda}} \cap \bm{\Pi}_{\rm cumber} | \bm{\pi}_k \not \in \bm{\Pi}^{\bm{\Lambda}}) \\ & \qquad \geq \prod_i (\kappa^i/|\Pi^i|) \prod_{n=k}^{k+\bar{L}-1} (1-\epsilon_n)\prod_i(1-\gamma_n^i).
\end{align*}
This leads to, for all $k\geq\tilde{k}$,
\begin{align*}
& \P( \bm{\pi}_{k+\bar{L}} \not\in \bm{\Pi}^{\bm{\Lambda}} \cap \bm{\Pi}_{\rm cumber}) \\ & \quad  \leq      \Big ( 1- \min \Big\{\bar{p}_{\min} \ , \ \prod_i (\kappa^i/|\Pi^i|) \Big \} \Big) \\ & \qquad  \times
\P( \bm{\pi}_k \not\in \bm{\Pi}^{\bm{\Lambda}} \cap \bm{\Pi}_{\rm cumber}) \\ & \qquad \quad + \sum_{n=k}^{k+\bar{L}-1} \epsilon_n + \sum_{n=k}^{k+\bar{L}-1} \sum_i \gamma_n^i.
\end{align*}
Since $\Big|1- \min \Big\{\bar{p}_{\min} , \prod_i \frac{\kappa^i}{|\Pi^i|} \Big \}\Big|<1$, $\sum_{k\in\mathbb{N}} \epsilon_k<\infty$, and $\sum_{i,k\in\mathbb{N}} \gamma_k^i<\infty$, we have $\sum_{j\in\mathbb{N}} \P( \bm{\pi}_{k+j \bar{L}} \not\in \bm{\Pi}^{\bm{\Lambda}} \cap \bm{\Pi}_{\rm cumber}) < \infty$, for all $k\in\mathbb{N}$. This results in $\sum_{k\in\mathbb{N}} \P( \bm{\pi}_k \not\in \bm{\Pi}^{\bm{\Lambda}} \cap \bm{\Pi}_{\rm cumber}) < \infty$.
Borel-Cantelli Lemma implies
$$\P( \bm{\pi}_k\not\in \bm{\Pi}^{\bm{\Lambda}}\cap \bm{\Pi}_{\rm cumber}, \ \textrm{for infinitely many} \ k\in\mathbb{N}) = 0.$$
Also $\sum_{k\in\mathbb{N}} \P( \BR_k^i \not = \BR^i(\bm{\pi}_k^{-i}))<\infty$, hence,
\begin{align*}
& \P( \BR_k^i \not = \BR^i(\bm{\pi}_k^{-i}), \ \textrm{for infinitely many} \ k\in\mathbb{N}) = 0.
\end{align*}

This proves the second part.

\item We have, for all $k\geq\tilde{k}$,
\begin{align*}
& \P( \bm{\pi}_{k+\bar{L}}\in \bm{\Pi}_{\rm cumber}^{\bm{\Lambda}} | \bm{\pi}_k\in \bm{\Pi}_{\rm cumber}^{\bm{\Lambda}} ) \\ & \qquad \geq \prod_{n=k}^{k+\bar{L}-1} (1-\epsilon_n) \prod_i (1-\max\{\gamma_n^i,\kappa^i\}).
\end{align*}
We also have, for all $k\geq\tilde{k}$,
\begin{align*}
& \P( \bm{\pi}_{k+\bar{L}}\in \bm{\Pi}_{\rm cumber}^{\bm{\Lambda}} | \bm{\pi}_k \not \in \bm{\Pi}_{\rm cumber}^{\bm{\Lambda}} ) \\ & \qquad \geq \bar{p}_{\min}\prod_{n=k}^{k+\bar{L}-1} (1-\epsilon_n) \prod_i (1-\max\{\gamma_n^i,\kappa^i\})
\end{align*}
where $\bar{p}_{\min}\in(0,1)$ is as in (\ref{eq:pmin}).
This leads to, for all $k\geq\tilde{k}$,
\begin{align*}
\P( \bm{\pi}_{k+\bar{L}} \not \in \bm{\Pi}_{\rm cumber}^{\bm{\Lambda}}) \leq  & \sum_{n=k}^{k+\bar{L}-1} \Big(\epsilon_n + \sum_i \max\{\gamma_n^i,\kappa^i\} \Big) \\ & + (1 - \bar{p}_{\min}) \P( \bm{\pi}_k \not \in \bm{\Pi}_{\rm cumber}^{\bm{\Lambda}}).
\end{align*}
Since $|1-\bar{p}_{\min}|<1$ and $\lim_{k\in\mathbb{N}} \sum_{n=k}^{k+\bar{L}-1} \epsilon_n =0$, we have, for all $k\in\mathbb{N}$,
\begin{align*}
& \limsup_{j\in\mathbb{N}} \P  ( \bm{\pi}_{k+j \bar{L}} \not \in \bm{\Pi}_{\rm cumber}^{\bm{\Lambda}})  \\
& \qquad \leq \limsup_{j\in\mathbb{N}} \sum_{n=k+j \bar{L}}^{k+(j+1)\bar{L}-1}  \sum_i \max\{\gamma_n^i,\kappa^i\}  / \bar{p}_{\min}.
\end{align*}

This proves the third part.

\item
It follows exactly the same as the third part by replacing $\bm{\Pi}_{\rm cumber}^{\bm{\Lambda}}$ with $\bm{\Pi}_{\rm cumber}$.
\end{enumerate}

%%%%%%%%%%%%%%%%%%%%%%%%%%%%%%%%%%%%

%\section*{Acknowledgments}

% \Cref{ex:secref} shows how to reference sections.

% \begin{example}[label={ex:secref},lefthand ratio=0.4,bicolor,sidebyside,%
% listing options={style=siamlatex,basicstyle=\ttfamily\scriptsize,%
% deletetexcs={cref,Cref},{moretexcs=[2]{cref,Cref}}}]
% {Right and wrong ways to reference a section}
% Inside a sentence\dots\\
% Single: \cref{sec:intro}\\
% Range: \cref{sec:intro,sec:front,%
% sec:sec}\\
% Multiple: \cref{sec:intro,sec:sec,%
% sec:tab,sec:math,sec:thm}\\
% Appendix: \cref{sec:changes}\\

% Beginning of a sentence\dots\\
% Single: \Cref{sec:intro}\\
% Range: \Cref{sec:intro,sec:front,%
% sec:sec}\\
% Multiple: \Cref{sec:intro,sec:sec,%
% sec:tab,sec:math,sec:thm}\\
% Appendix: \Cref{sec:changes}\\

% Just don't do it this way\dots\\
% Section~\ref{sec:intro}
% \end{example}

\bibliographystyle{plain}
\nocite{*}
\bibliography{YAY-TAC-bib}

\end{document}